\documentclass[twoside,english,american,DIV12,listtotoc,bibtotoc,idxtotoc]{scrartcl}
\usepackage[T1]{fontenc}
\usepackage[latin1]{inputenc}
\usepackage{babel}
\usepackage{prettyref}
\usepackage{amsthm}
\usepackage{amsmath}
\usepackage{amssymb}
\usepackage[unicode=true,pdfusetitle,
 bookmarks=true,bookmarksnumbered=false,bookmarksopen=false,
 breaklinks=false,pdfborder={0 0 1},backref=false,colorlinks=false]
 {hyperref}

\makeatletter
 \theoremstyle{definition}
 \newtheorem*{defn*}{\protect\definitionname}
\theoremstyle{plain}
\newtheorem{thm}{\protect\theoremname}[section]
  \theoremstyle{remark}
  \newtheorem{rem}[thm]{\protect\remarkname}
  \theoremstyle{plain}
  \newtheorem{lem}[thm]{\protect\lemmaname}
  \theoremstyle{plain}
  \newtheorem{prop}[thm]{\protect\propositionname}

\usepackage{helvet}
\usepackage[T1]{fontenc}

\usepackage{amsmath}
\usepackage{setspace}
\usepackage{amssymb}
\usepackage{enumerate}
\usepackage{url}

\makeatletter

\usepackage[T1]{fontenc}    

\usepackage{a4wide}        
\usepackage{tu-preprint}
\usepackage{amsmath}

\usepackage{euscript}
\usepackage{mathtools}
\allowdisplaybreaks[1] 
\usepackage{amsfonts}
\newenvironment{keywords}{ \noindent\footnotesize\textbf{Keywords and phrases:}}{}

\newenvironment{class}{\noindent\footnotesize\textbf{Mathematics subject classification 2010:}}{}

\theoremstyle{definition}
\newtheorem*{hyp}{Hypotheses}

\usepackage{color,tu-preprint}
\usepackage{prettyref}

\newrefformat{prop}{Proposition \ref{#1}}
\newrefformat{lem}{Lemma \ref{#1}}
\newrefformat{thm}{Theorem \ref{#1}}
\newrefformat{cor}{Corollary \ref{#1}}
\newrefformat{rem}{Remark \ref{#1}}
\newrefformat{exa}{Example \ref{#1}}
\newrefformat{def}{Definition \ref{#1}}
\newrefformat{eq}{(\ref{#1})}

\newcommand*{\trace}{\operatorname{trace}}

\newcommand*{\dive}{\operatorname{div}}

\newcommand*{\Grad}{\operatorname{Grad}}
\newcommand*{\Div}{\operatorname{Div}}

\newcommand*{\grad}{\operatorname{grad}}

\newcommand*{\supp}{\operatorname{supp}}

\renewcommand*{\i}{\mathrm{i}}
\newcommand{\Bullet}[1]{\begin{array}[b]{c} \bullet \vspace{-1mm}\\
\rule{3pt}{0pt}\hspace{-1.5mm} #1  \end{array}}

\renewcommand{\Re}{\operatorname{\mathfrak{Re}}}

\renewcommand{\tilde}{\widetilde}
\renewcommand*{\epsilon}{\varepsilon}
\renewcommand*{\rho}{\varrho}

\makeatother

\usepackage{babel}
\institut{Institut f\"ur Analysis}

\preprintnumber{MATH-AN-13-2012}

\preprinttitle{Autonomous Evolutionary Inclusions with Applications to Problems with Nonlinear Boundary Conditions.}

\author{Sascha Trostorff}

\AtBeginDocument{
  
}

\makeatother

  \addto\captionsamerican{\renewcommand{\definitionname}{Definition}}
  \addto\captionsamerican{\renewcommand{\lemmaname}{Lemma}}
  \addto\captionsamerican{\renewcommand{\propositionname}{Proposition}}
  \addto\captionsamerican{\renewcommand{\remarkname}{Remark}}
  \addto\captionsamerican{\renewcommand{\theoremname}{Theorem}}
  \addto\captionsenglish{\renewcommand{\definitionname}{Definition}}
  \addto\captionsenglish{\renewcommand{\lemmaname}{Lemma}}
  \addto\captionsenglish{\renewcommand{\propositionname}{Proposition}}
  \addto\captionsenglish{\renewcommand{\remarkname}{Remark}}
  \addto\captionsenglish{\renewcommand{\theoremname}{Theorem}}
  \providecommand{\definitionname}{Definition}
  \providecommand{\lemmaname}{Lemma}
  \providecommand{\propositionname}{Proposition}
  \providecommand{\remarkname}{Remark}
\providecommand{\theoremname}{Theorem}

\begin{document}
\makepreprinttitlepage

\author{ Sascha Trostorff \\ Institut f\"ur Analysis, Fachrichtung Mathematik\\ Technische Universit\"at Dresden\\ Germany\\ sascha.trostorff@tu-dresden.de}

\title{Autonomous Evolutionary Inclusions with Applications to Problems
with Nonlinear Boundary Conditions.}

\maketitle
\begin{abstract} \textbf{Abstract.} We study an abstract class of
autonomous differential inclusions in Hilbert spaces and show the
well-posedness and causality, by establishing the operators involved
as maximal monotone operators in time and space. Then the proof of
the well-posedness relies on a well-known perturbation result for
maximal monotone operators. Moreover, we show that certain types of
nonlinear boundary value problems are covered by this class of inclusions
and we derive necessary conditions on the operators on the boundary
in order to apply the solution theory. We exemplify our findings by
two examples. \end{abstract}


\begin{class} 34G25, 35F30, 35R20, 46N20, 47J35 \end{class}

\newpage

\tableofcontents{} 

\newpage

\section{Introduction}

In this article we study differential inclusions of the form 
\begin{equation}
(u,f)\in\partial_{0}M(\partial_{0}^{-1})+A,\label{eq:problem}
\end{equation}
where $\partial_{0}$ denotes the derivative with respect to time
and $A$ is assumed to be a maximal monotone relation in time and
space. The bounded linear operator $M(\partial_{0}^{-1})$, also acting
in time and space, is a so-called linear material law, as it was introduced
in \cite{Picard}. Since we are dealing with the autonomous case,
the operator $M(\partial_{0}^{-1})$ and the relation $A$ are assumed
to commute with the temporal translation operator. We will show that
under suitable conditions on $M(\partial_{0}^{-1})$ and $A$ the
problem \prettyref{eq:problem} is well-posed, i.e. Hadamard's requirements
on existence and uniqueness of a solution $u$ and its continuous
dependence on the given data $f$ are satisfied. Moreover the issue
of causality is addressed, meaning that the behaviour of the solution
up to a certain time $T$ just depends on the behaviour of the given
right hand side up to the same time $T$.\\
Originally this class of problems was discussed in \cite{Picard},
where the relation $A$ was given by a skew-selfadjoint spatial operator.
There, well-posedness and causality were shown under a positive-definiteness
constraint on the operator $M(\partial_{0}^{-1})$. Moreover, the
abstract solution theory was applied to several examples of linear
equations in mathematical physics. This solution theory was generalized
by the author to the case of $A$ being a maximal monotone relation
in space in \cite{Trostorff2012_NA}, where $M(\partial_{0}^{-1})$
was of the particular form $M_{0}+\partial_{0}^{-1}M_{1}.$ We will
adopt the proof for the solution theory presented in \cite{Trostorff2012_NA}
to show the well-posedness of problems of the form \prettyref{eq:problem}.
The proof mainly relies on the realization of the derivative $\partial_{0}$
as a maximal monotone operator on an exponentially weighted $L_{2}$-space
(see e.g. \cite{Picard_McGhee,picard1989hilbert} or Subsection 2.1
of this article) and the application of a well-known perturbation
result for maximal monotone relations (see \prettyref{prop:pert}).
For the theory of maximal monotone relations we refer to the monographs
\cite{Brezis,Morosanu,showalter_book,papageogiou,hu2000handbook}.
Inclusions of the form \prettyref{eq:problem} cover a large class
of possible evolutionary problems, such as integro-differential equations
\cite{Trostorff2012_integro}, delay and neutral differential equations
\cite{Kalauch2011}, certain types of control problems \cite{Picard2012_comprehensive_control,Picard2012_boundary_control}
and equations of mathematical physics involving hysteresis \cite{Trostorff2012_NA}.\\
The article is structured as follows. In Section 2 we recall the definition
of the time-derivative, of linear material laws and some basic facts
on maximal monotone relations on Hilbert spaces. Moreover, we recall
the notion of so-called abstract boundary data spaces (see \cite{Picard2012_comprehensive_control}),
which will enable us to formulate evolutionary equations with nonlinear
boundary condition as an inclusion of the form \prettyref{eq:problem}.
In Section 3 the solution theory for \prettyref{eq:problem} is presented.
In the remaining part of the article we consider a certain type of
partial differential equations with nonlinear boundary conditions.
Those problems occur for instance in frictional contact problems (see
e.g. \cite{do1985inequality,Sofonea2009,Migorski_2010,Migorski_2010_cylinder})
in the field of elasticity, where the boundary condition is given
by a differential inclusion. In Subsection 4.1 we study an abstract
nonlinear differential operator $A$, which is given by the restriction
of a linear operator to elements, which satisfy the nonlinear boundary
condition. We show that under certain constraints on the relation,
which occurs in the boundary condition, the operator $A$ gets maximal
monotone and thus, the solution theory developed in Section 3 applies.
We illustrate our results in Subsection 4.2. There, in a first example
we consider the wave equation with an impedance-type boundary condition.
This problem was originally considered in \cite{Picard2012_Impedance}
and we show that it fits in our framework. The second example deals
with the equations of visco-elasticity with a frictional boundary
condition, which is modelled by a differential inclusion on the boundary.
A similar kind of this problem was considered in \cite{Migorski_2009}
for a cylindrical domain and antiplane shear deformations. Again we
show that the equations are covered by our abstract inclusion \prettyref{eq:problem}.\\
Throughout every Hilbert space is assumed to be complex. Its inner
product and its norm are denoted by $\langle\cdot|\cdot\rangle$ and
$|\cdot|,$ respectively, where the inner product is assumed to be
linear with respect to the second and conjugate linear with respect
to the first argument. Moreover, for a Hilbert space $H$ and a closed
subspace $V\subseteq H$ we denote by $\pi_{V}:H\to V$ the orthogonal
projection onto $V.$%
\footnote{Note that then the adjoint $\pi_{V}^{\ast}:V\to H$ is just the canonical
embedding.%
} Then $P_{V}\coloneqq\pi_{V}^{\ast}\pi_{V}:H\to H$ is the orthogonal
projector on $V$.

\section{Preliminaries}

\subsection{The time derivative and linear material laws}

In this subsection we recall, how to establish the time-derivative
as a normal, boundedly invertible linear operator and we recall the
notion of linear material laws. For the proofs and more details we
refer to \cite{Picard,Kalauch2011,Picard_McGhee}. Let $H$ be a Hilbert
space. For $\nu>0$ we define $H_{\nu,0}(\mathbb{R};H)$ as the space
of measurable, $H$-valued functions on $\mathbb{R}$, which are square-integrable%
\footnote{Throughout we identify the equivalence classes with respect to equality
almost everywhere with their representatives.%
} with respect to the exponentially weighted Lebesgue-measure $\mu_{\nu}\coloneqq\exp(-2\nu\,\cdot\,)\lambda,$
where by $\lambda$ we denote the one-dimensional Lebesgue measure.
Moreover, we set $H_{\nu,0}(\mathbb{R})\coloneqq H_{\nu,0}(\mathbb{R};\mathbb{C}).$
We define the operator $\partial_{0,\nu}$ as the closure of the mapping
\begin{align*}
\partial_{0,\nu}|_{C_{c}^{\infty}(\mathbb{R})}:C_{c}^{\infty}(\mathbb{R})\subseteq H_{\nu,0}(\mathbb{R}) & \to H_{\nu,0}(\mathbb{R})\\
\phi & \mapsto\phi',
\end{align*}
where by $C_{c}^{\infty}(\mathbb{R})$ we denote the space of arbitrarily
differentiable functions with compact support in $\mathbb{R}.$ Indeed,
the operator $\partial_{0,\nu}$ turns out to be normal and continuously
invertible and its inverse is given by 
\begin{equation}
\left(\partial_{0,\nu}^{-1}u\right)(t)\coloneqq\intop_{-\infty}^{t}u(s)\mbox{ d}s\quad(t\in\mathbb{R})\label{eq:d0_inverse}
\end{equation}
for each $u\in H_{\nu,0}(\mathbb{R}).$ Moreover $\Re\partial_{0,\nu}=\nu$,
which gives $\|\partial_{0,\nu}^{-1}\|\leq\frac{1}{\nu}$. Indeed,
one can show that the operator norm of $\partial_{0,\nu}^{-1}$ is
equal to $\nu^{-1}$ (see e.g. \cite{Kalauch2011}). Moreover, equation
\prettyref{eq:d0_inverse} shows the causality%
\footnote{Let $E,F$ be vector spaces and $G:\mathcal{D}(G)\subseteq E^{\mathbb{R}}\to F^{\mathbb{R}}$.
$G$ is called \emph{causal }(cf. \cite{lakshmikantham2010theory}),\emph{
}if for each $a\in\mathbb{R}$ and $f,g\in\mathcal{D}(G)$ the implication
\[
\left(\chi_{(-\infty,a]}(m)\left(f-g\right)=0\right)\Rightarrow\left(\chi_{(-\infty,a]}(m)\left(G(f)-G(g)\right)=0\right)
\]
holds. By $\chi_{(-\infty,a]}(m)$ we denote the multiplication operator
with the cut-off function $\chi_{(-\infty,a]}$, i.e. $\left(\chi_{(-\infty,a]}(m)f\right)(t)\coloneqq\chi_{(-\infty,a]}(t)f(t).$ %
} of $\partial_{0,\nu}^{-1}$. \\
Furthermore, we can give a spectral representation for $\partial_{0,\nu}$
in terms of the so-called \emph{Fourier-Laplace-transform} $\mathcal{L}_{\nu}:H_{\nu,0}(\mathbb{R})\to L_{2}(\mathbb{R}),$
which is given by $\mathcal{L}_{\nu}\coloneqq\mathcal{F}\exp(-\nu m),$
where $\mathcal{F}$ denotes the Fourier-transform in $L_{2}(\mathbb{R})$
given by the unitary extension of 
\[
\left(\mathcal{F}\phi\right)(x)\coloneqq\frac{1}{\sqrt{2\pi}}\intop_{\mathbb{R}}\phi(y)\exp(-\i xy)\mbox{ d}y\quad(x\in\mathbb{R},\phi\in C_{c}^{\infty}(\mathbb{R}))
\]
and $\exp(-\nu m):H_{\nu,0}(\mathbb{R})\to L_{2}(\mathbb{R})$ is
defined by $\left(\exp(-\nu m)u\right)(t)=\exp(-\nu t)u(t).$ Clearly,
$\exp(-\nu m)$ is a unitary operator and so is $\mathcal{L}_{\nu}.$
Using the well-known spectral representation of the weak derivative
on $L_{2}(\mathbb{R})$ via the Fourier-transform $\mathcal{F}$,
we get that 
\begin{equation}
\partial_{\nu,0}=\mathcal{L}_{\nu}^{\ast}\left(\i m+\nu\right)\mathcal{L}_{\nu},\label{eq:spectral_derivative}
\end{equation}
where $m:\mathcal{D}(m)\subseteq L_{2}(\mathbb{R})\to L_{2}(\mathbb{R})$
denotes the multiplication-by-the-argument operator with maximal domain,
i.e. $\left(mf\right)(t)\coloneqq tf(t)$ for all $f\in\mathcal{D}(m)=\{g\in L_{2}(\mathbb{R})\,|\,(x\mapsto xg(x))\in L_{2}(\mathbb{R})\}$
and $t\in\mathbb{R}$. Note that the operators $\partial_{0,\nu}$
and $\mathcal{L}_{\nu}$ can be extended to $H_{\nu,0}(\mathbb{R};H)\cong H_{\nu,0}(\mathbb{R})\otimes H$
in the canonical way. \\
Next we consider a bounded strongly measurable mapping $M:B_{\mathbb{C}}\left(\frac{1}{2\nu},\frac{1}{2\nu}\right)\to L(H)$,
where $B_{\mathbb{C}}(x,r)$ denotes the open ball in $\mathbb{C}$
around the center $x\in\mathbb{C}$ with radius $r>0.$ We define
the bounded linear operator $M\left(\frac{1}{\i m+\nu}\right):L_{2}(\mathbb{R};H)\to L_{2}(\mathbb{R};H)$
by setting 
\[
\left(M\left(\frac{1}{\i m+\nu}\right)u\right)(t)\coloneqq M\left(\frac{1}{\i t+\nu}\right)u(t)\quad(t\in\mathbb{R}).
\]
The spectral representation \prettyref{eq:spectral_derivative} of
the derivative $\partial_{0,\nu}$ now leads to the definition of
the linear and bounded operator $M\left(\partial_{0,\nu}^{-1}\right)\coloneqq\mathcal{L}_{\nu}^{\ast}M\left(\frac{1}{\i m+\nu}\right)\mathcal{L}_{\nu}:H_{\nu,0}(\mathbb{R};H)\to H_{\nu,0}(\mathbb{R};H).$
If we assume additionally that $M$ is analytic, we call $M(\partial_{0,\nu}^{-1})$
a \emph{linear material-law}.%
\footnote{The analyticity yields, employing a Paley-Wiener-type result (cf.
\cite{rudin1987real}), the causality of the operator $M\left(\partial_{0,\nu}^{-1}\right).$%
} Moreover, as $M\left(\partial_{0,\nu}^{-1}\right)$ is a function
of $\partial_{0,\nu}^{-1}$, we obtain the translation invariance
of $M\left(\partial_{0,\nu}^{-1}\right),$ i.e. for each $h\in\mathbb{R}$
we have $\tau_{h}M\left(\partial_{0,\nu}^{-1}\right)=M\left(\partial_{0,\nu}^{-1}\right)\tau_{h}$,
where by $\tau_{h}$ we denote the translation operator on $H_{\nu,0}(\mathbb{R};H)$
given by $\left(\tau_{h}u\right)(t)\coloneqq u(t+h)$ for $u\in H_{\nu,0}(\mathbb{R};H)$
and $t\in\mathbb{R}.$

\subsection{Maximal monotone relations}

We recall the notion of monotone and maximal monotone relations in
Hilbert spaces. Moreover we state Minty's famous theorem, which gives
a characterization of maximal monotonicity, which will be frequently
used in the forthcoming parts. Additionally, we recall the definition
of the Yosida-approximation of a maximal monotone relation and state
a well-known perturbation result, which will be the key argument for
proving our solution theory for inclusions of the form \prettyref{eq:problem}.
Throughout let $H$ be a Hilbert space and $A\subseteq H\oplus H$
be a binary relation. 
\begin{defn*}
The relation $A$ is called \emph{monotone, }if for each $(u,v),(x,y)\in A$
the inequality 
\[
\Re\langle u-x|v-y\rangle\geq0
\]
holds. We call $A$ \emph{maximal monotone}, if $A$ is monotone and
there exists no proper monotone extension of $A$, i.e. for each monotone
relation $B\subseteq H\oplus H$ with $A\subseteq B$ it follows that
$A=B.$
\end{defn*}
From this definition we can immediately derive the \emph{demi-closedness}
of maximal monotone relations, which in particular means that if $(u_{n})_{n\in\mathbb{N}}$
and $(v_{n})_{n\in\mathbb{N}}$ are sequences in $H$ such that $u_{n}\rightharpoonup u$
and $v_{n}\to v$ and $(u_{n},v_{n})\in A$ for each $n\in\mathbb{N},$
then $(u,v)\in A.$ \\
Before we state Minty's Theorem, we introduce a linear structure on
the set of binary relations. Let $H_{0},H_{1}$ be two Hilbert spaces,
$B,C\subseteq H_{0}\oplus H_{1}$ and $\lambda\in\mathbb{C}.$ Then
we define the relation $\lambda B+C\subseteq H_{0}\oplus H_{1}$ by
\[
\lambda B+C\coloneqq\left\{ \left.(u,v)\in H_{0}\oplus H_{1}\,\right|\,\exists x,y\in H_{1}:(u,x)\in B,(u,y)\in C,v=\lambda x+y\right\} .
\]
Moreover, for $M\subseteq H_{0}$ and $N\subseteq H_{1}$ we set 
\begin{align*}
B[M] & \coloneqq\left\{ y\in H_{1}\,|\,\exists x\in M:(x,y)\in B\right\} ,\\
{}[N]B & \coloneqq\left\{ x\in H_{0}\,|\,\exists y\in N:(x,y)\in B\right\} .
\end{align*}
For a monotone relation we can characterize the maximal monotonicity
by using Minty's Theorem. A proof can be found for instance in \cite{Minty,Brezis,Trostorff_2011}.
\begin{thm}[G. Minty, \cite{Minty}]
\label{thm:Minty} Let $A$ be monotone. Then the following statements
are equivalent:

\begin{enumerate}[(i)]

\item $A$ is maximal monotone,

\item there exists $\lambda>0$ such that $\left(1+\lambda A\right)[H]=H,$

\item for each $\lambda>0$ it holds $\left(1+\lambda A\right)[H]=H.$

\end{enumerate}\end{thm}
\begin{rem}
\label{rem:maximal monotone} Let $A$ be maximal monotone.

\begin{enumerate}[(a)]

\item Due to the monotonicity of $A$, the relation $\left(1+\lambda A\right)^{-1}$
is right-unique for each $\lambda>0$, i.e. a mapping, whose domain
equals $H$ according to \prettyref{thm:Minty}. Indeed, one can show
that $J_{\lambda}(A)\coloneqq\left(1+\lambda A\right)^{-1}$ is Lipschitz-continuous
with a Lipschitz-constant less than or equal to $1$ (see e.g. \cite[Theorem 1.3]{Morosanu},
\cite[Proposition 1.12]{Trostorff_2011}). Moreover, we define $A_{\lambda}\coloneqq\lambda^{-1}\left(1-J_{\lambda}(A)\right)$
for $\lambda>0$ and obtain again a Lipschitz-continuous mapping with
a Lipschitz-constant less than or equal to $\lambda^{-1}$. Furthermore,
$A_{\lambda}$ is monotone for each $\lambda>0$. The mappings $\left(A_{\lambda}\right)_{\lambda>0}$
are called the \emph{Yosida-approximation} of $A.$ Moreover, for
each $\lambda>0$ and $x\in H$ we have $(J_{\lambda}(A)(x),A_{\lambda}(x))\in A.$

\item For $x\in[H]A$ the set $A[\{x\}]$ is convex and closed. Hence,
we can find the element in $A[\{x\}]$ possessing minimal norm, and
we denote this element by $A^{0}(x).$ Then one can show, that for
each $x\in[H]A$ and $\lambda>0$ the inequality $|A_{\lambda}(x)|\leq|A^{0}(x)|$
holds (see \cite[Theorem 1.3]{Morosanu}, \cite[Proposition 1.12]{Trostorff_2011}).

\end{enumerate}
\end{rem}
Finally, we want to state two perturbation results for maximal monotone
relations. If $A,B\subseteq H\oplus H$ are two monotone relations,
then clearly $A+B$ is also monotone. However, in general the sum
of two maximal monotone relations is not maximal monotone again. So
the question is: If $A$ and $B$ are maximal monotone, when is $A+B$
maximal monotone? A positive answer can be given, if one of the relations
is a Lipschitz-continuous mapping defined on the whole space $H.$ 
\begin{lem}
\label{lem:Yosida}Let $A$ be maximal monotone and $B:H\to H$ monotone
and Lipschitz-continuous. Then $A+B$ is maximal monotone.
\end{lem}
The proof is based on an application of the contraction mapping theorem
and can be found for instance in \cite[Lemme 2.4]{Brezis} or in \cite[Lemma 1.15]{Trostorff_2011}.
If $A$ and $B$ are maximal monotone relations in $H,$ we can use
the previous lemma and the fact that $B_{\lambda}$ is monotone and
Lipschitz-continuous for each $\lambda>0$ to derive the following
perturbation result.
\begin{prop}
\label{prop:pert}Let $A,B\subseteq H\oplus H$ be maximal monotone
and $y\in H.$ For $\lambda>0$ define $x_{\lambda}\coloneqq\left(1+A+B_{\lambda}\right)^{-1}(y).$
If $\sup_{\lambda>0}\left|B_{\lambda}(x_{\lambda})\right|<\infty,$
then $\left(x_{\lambda}\right)_{\lambda>0}$ converges to some $x\in H$
as $\lambda\to0$ with $(x,y)\in1+A+B.$ In particular $A+B$ is maximal
monotone.
\end{prop}
For a proof of this statement we refer to \cite[Lemma 7.41]{papageogiou}
or \cite[Lemma 1.17]{Trostorff_2011}.

\subsection{Abstract boundary data spaces}

The last concept we need for our main results are so-called abstract
boundary data spaces, introduced in \cite{Picard2012_comprehensive_control}.
Let $H_{0},H_{1}$ be two Hilbert spaces and $G_{c}\subseteq H_{0}\oplus H_{1},D_{c}\subseteq H_{1}\oplus H_{0}$
be two densely defined closed linear operators. Moreover, assume that
$G_{c}$ and $D_{c}$ are \emph{formally skew-adjoint}, i.e. $G_{c}\subseteq\left(-D_{c}\right)^{\ast}\eqqcolon G$,
which also yields $D_{c}\subseteq\left(-G_{c}\right)^{\ast}\eqqcolon D.$%
\footnote{As a typical example one can think of $G_{c}$ and $D_{c}$ as the
closure of the gradient and the divergence on $L_{2}$, defined on
test-functions. Then the adjoints $G$ and $D$ are nothing but the
usual weak gradient and weak divergence in $L_{2}$ (see Subsection
4.2). %
} If we equip the domains of $G$ and $D$ with their respective graph-norms,
we get that $\mathcal{D}(G_{c})$ and $\mathcal{D}(D_{c})$ are closed
linear subspaces of $\mathcal{D}(G)$ and $\mathcal{D}(D)$ and hence,
the projection theorem yields 
\begin{align*}
\mathcal{D}(G) & \coloneqq\mathcal{D}(G_{c})\oplus\mathcal{D}(G_{c})^{\bot},\\
\mathcal{D}(D) & \coloneqq\mathcal{D}(D_{c})\oplus\mathcal{D}(D_{c})^{\bot},
\end{align*}
where the orthocomplements have to be taken with respect to the inner
products induced by the graph norms of $G$ and $D$, respectively.
An easy computation yields 
\begin{align*}
\mathcal{BD}(G)\coloneqq\mathcal{D}(G_{c})^{\bot} & =[\{0\}](1-DG),\\
\mathcal{BD}(D)\coloneqq\mathcal{D}(D_{c})^{\bot} & =[\{0\}](1-GD).
\end{align*}
The spaces $\mathcal{BD}(G)$ and $\mathcal{BD}(D)$ are called \emph{abstract
boundary data spaces} of $G$ and $D,$ respectively.%
\footnote{In applications, where $G_{c}$ and $D_{c}$ will be defined as the
closure of differential operators, defined on test functions, the
elements of the domains of $G_{c}$ and $D_{c}$ can be interpreted
as those elements in the domain of $G$ and $D$ with vanishing traces.
For instance, if $G_{c}$ is the closure of the gradient defined on
test-functions, then $\mathcal{D}(G_{c})$ is the classical Sobolev-space
$W_{2,0}^{1}$, i.e. the space of weakly differentiable functions
in $L_{2}$ with vanishing traces. Hence, the boundary values of a
function in $\mathcal{D}(G)$ just depend on the boundary values of
its projection onto $\mathcal{BD}(G).$ %
} Clearly, $G[\mathcal{BD}(G)]\subseteq\mathcal{BD}(D)$ and $D[\mathcal{BD}(D)]\subseteq\mathcal{BD}(G).$
We define the operators 
\begin{align*}
\Bullet G:\mathcal{BD}(G) & \to\mathcal{BD}(D),\\
\Bullet D:\mathcal{BD}(D) & \to\mathcal{BD}(G),
\end{align*}
as the respective restrictions of $G$ and $D$. As $\mathcal{BD}(G)$
and $\mathcal{BD}(D)$ are closed subspaces of $\mathcal{D}(G)$ and
$\mathcal{D}(D),$ respectively, they inherit the Hilbert space structure
from these supersets. With respect to these topologies, the operators
$\Bullet G$ and $\Bullet D$ enjoy a surprising property.
\begin{prop}[{\cite[Theorem 5.2]{Picard2012_comprehensive_control}}]
 The operators $\Bullet G$ and $\Bullet D$ are unitary. Furthermore,
$\left(\Bullet G\right)^{\ast}=\Bullet D$ and so $\left(\Bullet D\right)^{\ast}=\Bullet G.$ 
\end{prop}

\section{Evolutionary inclusions}

Let $H$ be a Hilbert space. In this section we study evolutionary
inclusions of the form 
\begin{equation}
(u,f)\in\overline{\partial_{0,\nu}M(\partial_{0,\nu}^{-1})+A},\label{eq:evol}
\end{equation}
where $M:B_{\mathbb{C}}\left(\frac{1}{2\nu},\frac{1}{2\nu}\right)\to L(H)$
is a linear material law for some $\nu>0$ and $A\subseteq H_{\nu,0}(\mathbb{R};H)\oplus H_{\nu,0}(\mathbb{R};H)$
is an autonomous maximal monotone relation, i.e. $A$ is maximal monotone
and for each $h\in\mathbb{R}$ and $(u,v)\in A$ it follows that $\left(u(\cdot+h),v(\cdot+h)\right)\in A.$
The function $f\in H_{\nu,0}(\mathbb{R};H)$ is a given source term
and $u\in H_{\nu,0}(\mathbb{R};H)$ is the unknown. We address the
question of well-posedness of this problem, i.e. we show the uniqueness,
existence and continuous dependence of a solution $u$ on the given
data $f.$ More precisely, we show that $\left(\overline{\partial_{0,\nu}M(\partial_{0,\nu}^{-1})+A}\right)^{-1}$
is a Lipschitz-continuous mapping, whose domain is the whole space
$H_{\nu,0}(\mathbb{R};H).$ The second property we consider, is causality
of the \emph{solution operator $\left(\overline{\partial_{0,\nu}M(\partial_{0,\nu}^{-1})+A}\right)^{-1}.$ }

\subsection{Well-posedness}

This subsection is devoted to the well-posedness of differential inclusions
of the form \prettyref{eq:evol}. Throughout, let $H$ be a Hilbert
space, $M:B_{\mathbb{C}}\left(\frac{1}{2\nu},\frac{1}{2\nu}\right)\to L(H)$
a bounded strongly measurable function for some $\nu>0$ and $A\subseteq H_{\nu,0}(\mathbb{R};H)\oplus H_{\nu,0}(\mathbb{R};H)$. 

\begin{hyp} We say that $M$ and $A$ satisfy the hypotheses (H1),
(H2) and (H3) respectively, if 

\begin{enumerate}[(H1)]

\item there exists $c>0$ such that for all $z\in B_{\mathbb{C}}\left(\frac{1}{2\nu},\frac{1}{2\nu}\right)$
the inequality $\Re z^{-1}M(z)\geq c$ holds.

\item $A$ is maximal monotone and \emph{autonomous}, i.e. for every
$h\in\mathbb{R}$ and $(u,v)\in A$ we have $\left(u(\cdot+h),v(\cdot+h)\right)\in A.$

\item for all $(u,v),(x,y)\in A$ the estimate $\intop_{-\infty}^{0}\Re\langle u(t)-x(t)|v(t)-y(t)\rangle e^{-2\nu t}\mbox{d}t\geq0$
holds.

\end{enumerate}

\end{hyp}
\begin{rem}
\label{rem:autonomous relations}$\,$

\begin{enumerate}[(a)]

\item A typical example for an autonomous maximal monotone relation
is the extension of a maximal monotone relation $B\subseteq H\oplus H,$
given by 
\[
A\coloneqq\{(u,v)\in H_{\nu,0}(\mathbb{R};H)\oplus H_{\nu,0}(\mathbb{R};H)\,|\,\left(u(t),v(t)\right)\in B\mbox{ for almost every }t\in\mathbb{R}\}.
\]
If one assumes that $(0,0)\in B,$ then $A$ is maximal monotone (cf.
\cite[p. 31]{Morosanu}).

\item In case of $A$ being the canonical extension of a skew-selfadjoint
operator on $H$, the evolutionary problem \prettyref{eq:evol} was
originally considered by Picard in several works (cf. \cite{Picard,Picard2010,Picard_McGhee}).
The more general case of an autonomous operator $A:D(A)\subseteq H_{\nu,0}(\mathbb{R};H)\to H_{\nu,0}(\mathbb{R};H)$
was treated in \cite{Picard2012_Impedance}. 

\item If $A$ is the extension of a maximal monotone relation $B\subseteq H\oplus H$
(compare (a)) and $M(\partial_{0,\nu}^{-1})$ is of the particular
form $M(\partial_{0,\nu}^{-1})=M_{0}+\partial_{0,\nu}^{-1}M_{1}$,
the problem was addressed by the author in \cite{Trostorff2012_NA}.
Indeed, in \cite{Trostorff2012_NA} the inclusion was studied on the
half line $\mathbb{R}_{\geq0}$ instead of $\mathbb{R},$ which in
particular implies that one can omit the additional assumption $(0,0)\in B$
in (a) in order to obtain the maximal monotonicity of the extension
$A.$ 

\end{enumerate}
\end{rem}
First, we state the following simple, but useful observation.
\begin{lem}
\label{lem:core}Let $\nu>0$ and $T\in L(H_{\nu,0}(\mathbb{R};H))$.
Assume that $T\partial_{0,\nu}\subseteq\partial_{0,\nu}T$. Then $\mathcal{D}(\partial_{0,\nu})$
is a core of $\partial_{0,\nu}T$ and $\left(\partial_{0,\nu}T\right)^{\ast}=\partial_{0,\nu}^{\ast}T^{\ast}.$\end{lem}
\begin{proof}
Let $u\in\mathcal{D}(\partial_{0,\nu}T),$ i.e. $Tu\in\mathcal{D}(\partial_{0,\nu}).$
For $\varepsilon>0$ we define $u_{\varepsilon}\coloneqq(1+\varepsilon\partial_{0,\nu})^{-1}u\in\mathcal{D}(\partial_{0,\nu})$
and we obtain that \foreignlanguage{english}{$u_{\varepsilon}\to u$}
as $\varepsilon\to0$ as well as 
\[
\partial_{0,\nu}Tu_{\varepsilon}=\partial_{0,\nu}T(1+\varepsilon\partial_{0,\nu})^{-1}u=(1+\varepsilon\partial_{0,\nu})^{-1}\partial_{0,\nu}Tu\to\partial_{0,\nu}Tu\quad(\varepsilon\to0).
\]
In particular, this yields that $\partial_{0,\nu}T=\overline{T\partial_{0,\nu}}$
and hence, 
\[
\left(\partial_{0,\nu}T\right)^{\ast}=\left(T\partial_{0,\nu}\right)^{\ast}=\partial_{0,\nu}^{\ast}T^{\ast}.\tag*{\qedhere}
\]
\end{proof}
\begin{lem}
\label{lem:maximal monoton material law} Assume that $M$ satisfies
(H1). Then $\partial_{0,\nu}M(\partial_{0,\nu}^{-1})-c$ is a maximal
monotone operator in $H_{\nu,0}(\mathbb{R};H).$\end{lem}
\begin{proof}
Using the unitary equivalence of $\partial_{0,\nu}M(\partial_{0,\nu}^{-1})$
in $H_{\nu,0}(\mathbb{R};H)$ and $\left(\i m+\nu\right)$$M\left(\frac{1}{\i m+\nu}\right)$
in $L_{2}(\mathbb{R};H)$, where $m$ denotes the multiplication-by-the-argument
operator (see Subsection 2.1), it suffices to prove the maximal monotonicity
of $\left(\i m+\nu\right)M\left(\frac{1}{\i m+\nu}\right)-c$ in $L_{2}(\mathbb{R};H).$
The monotonicity follows immediately from (H1). Furthermore, by (H1)
we get the strict monotonicity of $\left(\left(\i m+\nu\right)M\left(\frac{1}{\i m+\nu}\right)\right)^{\ast}=\left(-\i m+\nu\right)M\left(\frac{1}{\i m+\nu}\right)^{\ast}$,
where for $u\in L_{2}(\mathbb{R};H)$ and almost every $t\in\mathbb{R}$
we have that $\left(M\left(\frac{1}{\i m+\nu}\right)^{\ast}u\right)(t)=M\left(\frac{1}{\i t+\nu}\right)^{\ast}u(t).$
The latter implies that the operator $\left(\i m+\nu\right)M\left(\frac{1}{\i m+\nu}\right)$
has a dense range, which gives, using the closedness and the continuous
invertibility of the operator, that $\left(\i m+\nu\right)M\left(\frac{1}{\i m+\nu}\right)$
is onto. Hence \prettyref{thm:Minty} yields the assertion.\end{proof}
\begin{prop}
\label{prop:uniq} Let $B\subseteq H_{\nu,0}(\mathbb{R};H)\oplus H_{\nu,0}(\mathbb{R};H)$
be monotone and let $M$ satisfy (H1). Then for each $(u,f),(v,g)\in\partial_{0,\nu}M\left(\partial_{0,\nu}^{-1}\right)+B$
the estimate 
\[
|u-v|_{H_{\nu,0}(\mathbb{R};H)}\leq\frac{1}{c}|f-g|_{H_{\nu,0}(\mathbb{R};H)}
\]
holds.\end{prop}
\begin{proof}
Since the operator $\partial_{0,\nu}M\left(\partial_{0,\nu}^{-1}\right)-c$
is monotone, according to \prettyref{lem:maximal monoton material law},
we obtain that $\partial_{0,\nu}M\left(\partial_{0,\nu}^{-1}\right)-c+B$
is monotone. An easy argument, using the Cauchy-Schwarz-inequality,
yields the assertion.\end{proof}
\begin{rem}
\prettyref{prop:uniq} especially yields the uniqueness and continuous
dependence of a solution $u$ of \prettyref{eq:evol} on the given
right hand side $f$.
\end{rem}
The next proposition yields the existence of a solution of \prettyref{eq:evol}
for every right hand side $f\in H_{\nu,0}(\mathbb{R};H)$. The proof
uses the perturbation result given in \prettyref{prop:pert} in order
to show the existence of a solution for $f\in C_{c}^{\infty}(\mathbb{R};H)$
and follows the same strategy as the one in \cite[Proposition 4.6]{Trostorff2012_NA}.
\begin{prop}
\label{prop:existence}Let $M$ satisfy (H1) and $A$ satisfy (H2).
Then for each $f\in C_{c}^{\infty}(\mathbb{R};H)$ there exists $u\in H_{\nu,0}(\mathbb{R};H)$
such that 
\[
(u,f)\in\partial_{0,\nu}M\left(\partial_{0,\nu}^{-1}\right)+A.
\]
\end{prop}
\begin{proof}
We replace $A$ by its Yosida approximation $A_{\lambda}$ for $\lambda>0$,
which is a Lipschitz-continuous monotone mapping on $H_{\nu,0}(\mathbb{R};H).$
Then, using \prettyref{lem:maximal monoton material law} and \prettyref{lem:Yosida}
we find an element $u_{\lambda}\in H_{\nu,0}(\mathbb{R};H)$ such
that 
\[
\left(u_{\lambda},\frac{1}{c}f\right)\in1+\frac{1}{c}\left(\partial_{0,\nu}M\left(\partial_{0,\nu}^{-1}\right)-c+A_{\lambda}\right),
\]
which is equivalent to 
\begin{equation}
(u_{\lambda},f)\in\partial_{0,\nu}M\left(\partial_{0,\nu}^{-1}\right)+A_{\lambda}.\label{eq:appr_problem}
\end{equation}
We prove that $\left(A_{\lambda}\left(u_{\lambda}\right)\right)_{\lambda>0}$
is uniformly bounded. For that purpose let $h>0.$ Then, using that
$A_{\lambda}$ and $\partial_{0,\nu}M\left(\partial_{0,\nu}^{-1}\right)$
commute with the operator $\tau_{h}:H_{\nu,0}(\mathbb{R};H)\to H_{\nu,0}(\mathbb{R};H)$
given by $\left(\tau_{h}v\right)(t)\coloneqq v(t+h)$ for $t\in\mathbb{R},\, v\in H_{\nu,0}(\mathbb{R};H),$
we obtain that 
\[
(\tau_{h}u_{\lambda},\tau_{h}f)\in\partial_{0,\nu}M\left(\partial_{0,\nu}^{-1}\right)+A_{\lambda}.
\]
Since $A_{\lambda}$ is monotone, \prettyref{prop:uniq} yields 
\[
|\tau_{h}u_{\lambda}-u_{\lambda}|_{H_{\nu,0}(\mathbb{R};H)}\leq\frac{1}{c}|\tau_{h}f-f|_{H_{\nu,0}(\mathbb{R};H)}
\]
and hence, using the mean-value inequality 
\[
\left|\frac{1}{h}\left(\tau_{h}u_{\lambda}-u_{\lambda}\right)\right|_{H_{\nu,0}(\mathbb{R};H)}\leq\frac{1}{c}|f'|_{\infty}\sqrt{\mu_{\nu}\left(\supp f\right)}
\]
for each $h>0,$ where $\supp f$ denotes the support of $f$. The
latter gives, by choosing a weak-convergent subsequence of $\left(\frac{1}{h}\left(\tau_{h}u_{\lambda}-u_{\lambda}\right)\right)_{h>0}$,
that $u_{\lambda}\in D(\partial_{0,\nu})$ and that 
\[
|\partial_{0,\nu}u_{\lambda}|_{H_{\nu,0}(\mathbb{R};H)}\leq\frac{1}{c}|f'|_{\infty}\sqrt{\mu_{\nu}\left(\supp f\right).}
\]
Hence, using \prettyref{eq:appr_problem} we obtain 
\begin{align*}
\left|A_{\lambda}(u_{\lambda})\right|_{H_{\nu,0}(\mathbb{R};H)} & =\left|f-\partial_{0,\nu}M\left(\partial_{0,\nu}^{-1}\right)u_{\lambda}\right|_{H_{\nu,0}(\mathbb{R};H)}\\
 & \leq|f|_{H_{\nu,0}(\mathbb{R};H)}+\left\Vert M\left(\partial_{0,\nu}^{-1}\right)\right\Vert \frac{1}{c}|f'|_{\infty}\sqrt{\mu_{\nu}\left(\supp f\right)}
\end{align*}
for each $\lambda>0.$ Thus, \prettyref{prop:pert} applies and hence,
we find an element $u\in H_{\nu,0}(\mathbb{R};H)$ satisfying 
\[
\left(u,\frac{1}{c}f\right)\in1+\frac{1}{c}\left(\partial_{0,\nu}M\left(\partial_{0,\nu}^{-1}\right)-c+A\right),
\]
which gives 
\[
(u,f)\in\partial_{0,\nu}M\left(\partial_{0,\nu}^{-1}\right)+A.\tag*{\qedhere}
\]

\end{proof}
We summarize our findings of this subsection in the following theorem.
\begin{thm}[Well-posedness of evolutionary inclusions]
\label{thm:well_posedness} Let $H$ be a Hilbert space, $M:B_{\mathbb{C}}\left(\frac{1}{2\nu},\frac{1}{2\nu}\right)\to L(H)$
a bounded strongly measurable function for some $\nu>0$ satisfying
(H1) and $A\subseteq H_{\nu,0}(\mathbb{R};H)\oplus H_{\nu,0}(\mathbb{R};H)$
a relation satisfying (H2). Then for each $f\in H_{\nu,0}(\mathbb{R};H)$
there exists a unique $u\in H_{\nu,0}(\mathbb{R};H)$ such that 
\[
(u,f)\in\overline{\partial_{0,\nu}M\left(\partial_{0,\nu}^{-1}\right)+A}.
\]
Moreover, $\left(\overline{\partial_{0,\nu}M\left(\partial_{0,\nu}^{-1}\right)+A}\right)^{-1}$
is Lipschitz-continuous with a Lipschitz constant less than or equal
to $\frac{1}{c}.$ 
\end{thm}

\subsection{Causality}

In this subsection we prove the causality of the solution operator
$\left(\overline{\partial_{0,\nu}M\left(\partial_{0,\nu}^{-1}\right)+A}\right)^{-1}$
of an inclusion of the form \prettyref{eq:evol}. First we give an
equivalent condition for the causality of the operator $\left(\partial_{0,\nu}M\left(\partial_{0,\nu}^{-1}\right)\right)^{-1}$,
which will enable us to prove the causality for the case of non-vanishing
$A.$ The statement was already given in \cite[Proposition 2.65]{Trostorff_2011}
in a slightly more general version. However, for sake of completeness
we present the proof.
\begin{lem}
\label{lem:eq_causality}Let $\nu>0$ and $M:B_{\mathbb{C}}\left(\frac{1}{2\nu},\frac{1}{2\nu}\right)\to L(H)$
be a strongly measurable bounded mapping satisfying the hypothesis
(H1). Then the following statements are equivalent:

\begin{enumerate}[(i)]

\item the operator $\left(\partial_{0,\nu}M\left(\partial_{0,\nu}^{-1}\right)\right)^{-1}:H_{\nu,0}(\mathbb{R};H)\to H_{\nu,0}(\mathbb{R};H)$
is causal,

\item for every $u\in\mathcal{D}\left(\partial_{0,\nu}M\left(\partial_{0,\nu}^{-1}\right)\right)$
we have that 
\[
\intop_{-\infty}^{0}\Re\left\langle \left.\partial_{0,\nu}M\left(\partial_{0,\nu}^{-1}\right)u(t)\right|u(t)\right\rangle e^{-2\nu t}\,\mathrm{d}t\geq c\intop_{-\infty}^{0}\langle u(t)|u(t)\rangle e^{-2\nu t}\,\mathrm{d}t.
\]

\end{enumerate}\end{lem}
\begin{proof}
Assume that $\left(\partial_{0,\nu}M\left(\partial_{0,\nu}^{-1}\right)\right)^{-1}$
is causal and let $u\in\mathcal{D}\left(\partial_{0,\nu}M\left(\partial_{0,\nu}^{-1}\right)\right).$
We set $v\coloneqq\partial_{0,\nu}M\left(\partial_{0,\nu}^{-1}\right)u$
and obtain 
\begin{align*}
 & \intop_{-\infty}^{0}\Re\left\langle \left.\partial_{0,\nu}M\left(\partial_{0,\nu}^{-1}\right)u(t)\right|u(t)\right\rangle e^{-2\nu t}\mbox{ d}t\\
 & =\intop_{-\infty}^{0}\Re\left\langle v(t)\left|\left(\left(\partial_{0,\nu}M\left(\partial_{0,\nu}^{-1}\right)\right)^{-1}v\right)(t)\right.\right\rangle e^{-2\nu t}\mbox{ d}t\\
 & =\Re\left\langle v\left|\chi_{(-\infty,0]}(m)\left(\partial_{0,\nu}M\left(\partial_{0,\nu}^{-1}\right)\right)^{-1}v\right.\right\rangle _{H_{\nu,0}(\mathbb{R};H)}\\
 & =\Re\left\langle v\left|\chi_{(-\infty,0]}(m)\left(\partial_{0,\nu}M\left(\partial_{0,\nu}^{-1}\right)\right)^{-1}\left(\chi_{(-\infty,0]}(m)v\right)\right.\right\rangle _{H_{\nu,0}(\mathbb{R};H)}\\
 & =\Re\left\langle \chi_{(-\infty,0]}(m)v\left|\left(\partial_{0,\nu}M\left(\partial_{0,\nu}^{-1}\right)\right)^{-1}\left(\chi_{(-\infty,0]}(m)v\right)\right.\right\rangle _{H_{\nu,0}(\mathbb{R};H)}.
\end{align*}
Using \prettyref{lem:maximal monoton material law} we get that 
\begin{align*}
 & \intop_{-\infty}^{0}\Re\left\langle \left.\partial_{0,\nu}M\left(\partial_{0,\nu}^{-1}\right)u(t)\right|u(t)\right\rangle e^{-2\nu t}\mbox{ d}t\\
 & =\Re\left\langle \chi_{(-\infty,0]}(m)v\left|\left(\partial_{0,\nu}M\left(\partial_{0,\nu}^{-1}\right)\right)^{-1}\left(\chi_{(-\infty,0]}(m)v\right)\right.\right\rangle _{H_{\nu,0}(\mathbb{R};H)}\\
 & \geq c\left|\left(\partial_{0,\nu}M\left(\partial_{0,\nu}^{-1}\right)\right)^{-1}\left(\chi_{(-\infty,0]}(m)v\right)\right|_{H_{\nu,0}(\mathbb{R};H)}^{2}\\
 & \geq c\left|\chi_{(-\infty,0]}(m)\left(\partial_{0,\nu}M\left(\partial_{0,\nu}^{-1}\right)\right)^{-1}\left(\chi_{(-\infty,0]}(m)v\right)\right|_{H_{\nu,0}(\mathbb{R};H)}^{2}\\
 & =c\left|\chi_{(-\infty,0]}(m)\left(\partial_{0,\nu}M\left(\partial_{0,\nu}^{-1}\right)\right)^{-1}v\right|_{H_{\nu,0}(\mathbb{R};H)}^{2}\\
 & =c\intop_{-\infty}^{0}|u(t)|^{2}e^{-2\nu t}\mbox{ d}t.
\end{align*}
Assume now that (ii) holds. Using that $\partial_{0,\nu}M(\partial_{0,\nu}^{-1})$
is translation invariant, i.e. it commutes with the translation operator
$\tau_{h}$ for each $h\in\mathbb{R},$ the asserted inequality implies
\begin{equation}
\intop_{-\infty}^{a}\Re\left\langle \left.\partial_{0,\nu}M\left(\partial_{0,\nu}^{-1}\right)u(t)\right|u(t)\right\rangle e^{-2\nu t}\mbox{ d}t\geq c\intop_{-\infty}^{a}|u(t)|^{2}e^{-2\nu t}\mbox{ d}t\label{eq:caus}
\end{equation}
for each $a\in\mathbb{R},u\in\mathcal{D}\left(\partial_{0,\nu}M\left(\partial_{0,\nu}^{-1}\right)\right).$
Due to the linearity of $\left(\partial_{0,\nu}M\left(\partial_{0,\nu}^{-1}\right)\right)^{-1}$
it suffices to prove that $\chi_{(-\infty,a]}(m)f=0$ implies $\chi_{(-\infty,a]}(m)\left(\partial_{0,\nu}M\left(\partial_{0,\nu}^{-1}\right)\right)^{-1}f=0$
for each $f\in H_{\nu,0}(\mathbb{R};H).$ So let $f\in H_{\nu,0}(\mathbb{R};H)$
with $\chi_{(-\infty,a]}(m)f=0$ for some $a\in\mathbb{R}$ and define
$u\coloneqq\left(\partial_{0,\nu}M\left(\partial_{0,\nu}^{-1}\right)\right)^{-1}f.$
Then \prettyref{eq:caus} gives 
\[
c\intop_{-\infty}^{a}|u(t)|^{2}e^{-2\nu t}\mbox{ d}t\leq\intop_{-\infty}^{a}\Re\left\langle f(t)|u(t)\right\rangle e^{-2\nu t}\mbox{ d}t=0,
\]
which implies $\chi_{(-\infty,a]}(m)u=0.$\end{proof}
\begin{rem}
If $M:B_{\mathbb{C}}\left(\frac{1}{2\nu},\frac{1}{2\nu}\right)\to L(H)$
is a linear material law and satisfies the hypothesis (H1), then \prettyref{lem:eq_causality}
gives that \prettyref{eq:caus} holds for every $u\in\mathcal{D}\left(\partial_{0,\nu}M\left(\partial_{0,\nu}^{-1}\right)\right)$
and every $a\in\mathbb{R},$ since the assumed analyticity for $M$
yields the causality of $\left(\partial_{0,\nu}M\left(\partial_{0,\nu}^{-1}\right)\right)^{-1}$
(see \cite[Lemma 2.5]{Picard}).
\end{rem}
Now we are able to show the causality of the solution operator associated
to the evolutionary inclusion \prettyref{eq:evol}.
\begin{prop}[Causality of evolutionary inclusions]
\label{prop:causality} Let $H$ be a Hilbert space, $\nu>0$ and
$M:B_{\mathbb{C}}\left(\frac{1}{2\nu},\frac{1}{2\nu}\right)\to L(H)$
a linear material law, satisfying hypothesis (H1). Furthermore let
$A\subseteq H_{\nu,0}(\mathbb{R};H)\oplus H_{\nu,0}(\mathbb{R};H)$
be a binary relation satisfying (H2) and (H3). Then, the solution
operator 
\[
\left(\overline{\partial_{0,\nu}M\left(\partial_{0,\nu}^{-1}\right)+A}\right)^{-1}:H_{\nu,0}(\mathbb{R};H)\to H_{\nu,0}(\mathbb{R};H),
\]

which exists according to \prettyref{thm:well_posedness}, is causal. \end{prop}
\begin{proof}
First note that due to the translation invariance of $A$ and $\partial_{0,\nu}M\left(\partial_{0,\nu}^{-1}\right)$
it suffices to check that for $f,g\in H_{\nu,0}(\mathbb{R};H)$ satisfying
$\chi_{(-\infty,0]}(m)(f-g)=0$ we have that 
\[
\chi_{(-\infty,0]}(m)\left(\left(\overline{\partial_{0,\nu}M\left(\partial_{0,\nu}^{-1}\right)+A}\right)^{-1}(f)-\left(\overline{\partial_{0,\nu}M\left(\partial_{0,\nu}^{-1}\right)+A}\right)^{-1}(g)\right)=0.
\]
So let $f,g\in H_{\nu,0}(\mathbb{R};H)$ with $\chi_{(-\infty,0]}(m)(f-g)=0.$
We choose sequences $(f_{n})_{n\in\mathbb{N}}$ and $(g_{n})_{n\in\mathbb{N}}$
in $C_{c}^{\infty}(\mathbb{R};H)\subseteq\left(\partial_{0,\nu}M\left(\partial_{0,\nu}^{-1}\right)+A\right)[H_{\nu,0}(\mathbb{R};H)]$
(see \prettyref{prop:existence}) such that $f_{n}\to f$ and $g_{n}\to g$
in $H_{\nu,0}(\mathbb{R};H)$ as $n\to\infty.$ For $n\in\mathbb{N}$
we define 
\begin{align*}
u_{n} & \coloneqq\left(\partial_{0,\nu}M\left(\partial_{0,\nu}^{-1}\right)+A\right)^{-1}(f_{n}),\\
v_{n} & \coloneqq\left(\partial_{0,\nu}M\left(\partial_{0,\nu}^{-1}\right)+A\right)^{-1}(g_{n}),
\end{align*}
and due to the continuity of $\left(\overline{\partial_{0,\nu}M\left(\partial_{0,\nu}^{-1}\right)+A}\right)^{-1}$
we get that $u_{n}\to\left(\overline{\partial_{0,\nu}M\left(\partial_{0,\nu}^{-1}\right)+A}\right)^{-1}(f)$
as well as $v_{n}\to\left(\overline{\partial_{0,\nu}M\left(\partial_{0,\nu}^{-1}\right)+A}\right)^{-1}(g)$
in $H_{\nu,0}(\mathbb{R};H).$ We estimate, using \prettyref{lem:eq_causality}
and (H3) 
\[
\intop_{-\infty}^{0}\Re\langle u_{n}(t)-v_{n}(t)|f_{n}(t)-g_{n}(t)\rangle e^{-2\nu t}\mbox{ d}t\geq c\intop_{-\infty}^{0}|u_{n}(t)-v_{n}(t)|^{2}e^{-2\nu t}\mbox{ d}t
\]
for each $n\in\mathbb{N}.$ Passing to the limit as $n\to\infty$
gives 
\begin{align*}
\intop_{-\infty}^{0}\left|\left(\overline{\partial_{0,\nu}M\left(\partial_{0,\nu}^{-1}\right)+A}\right)^{-1}(f)-\left(\overline{\partial_{0,\nu}M\left(\partial_{0,\nu}^{-1}\right)+A}\right)^{-1}(g)\right|^{2}e^{-2\nu t}\mbox{ d}t & \leq0,
\end{align*}
which yields the causality. 
\end{proof}

\section{Nonlinear boundary conditions}

In this subsection we study a class of evolutionary equations, which
covers evolutionary problems with nonlinear boundary conditions, modelled
by an inclusion on the boundary. Those equations occur for instance
in the study of problems in elasticity with frictional boundary conditions,
where the behaviour on the boundary is described by variational inequalities
(cf. \cite{do1985inequality,Sofonea2009,naniewicz_1989}) or by suitable
differential inclusions (cf. \cite{Migorski_2009,Migorski_2010,Migorski_2010_cylinder}).
Moreover, since we consider operators acting in time and space, also
problems with boundary conditions given by ordinary differential equations,
delay equations or inclusions can be treated within our framework.\\
Throughout let $H_{0}$ and $H_{1}$ be Hilbert spaces and $G_{c}\subseteq H_{0}\oplus H_{1}$
and $D_{c}\subseteq H_{1}\oplus H_{0}$ two densely defined closed
linear operators. Moreover assume that $G_{c}$ and $D_{c}$ are formally
skew-adjoint, i.e. 
\begin{align*}
G_{c} & \subseteq-(D_{c})^{\ast}\eqqcolon G,\\
D_{c} & \subseteq-(G_{c})^{\ast}\eqqcolon D.
\end{align*}

We do not distinguish notationally between the operators $G_{c},D_{c},G$
and $D$ and their canonical extensions to $H_{\nu,0}(\mathbb{R};H_{0})$
and $H_{\nu,0}(\mathbb{R};H_{1}),$ respectively. Recall from Subsection
2.3 that 
\begin{align*}
\mathcal{BD}(G)\coloneqq\mathcal{D}(G_{c})^{\bot} & =[\{0\}](1-DG),\\
\mathcal{BD}(D)\coloneqq\mathcal{D}(D_{c})^{\bot} & =[\{0\}](1-GD),
\end{align*}
where the orthocomplements are taken with respect to the inner products
in $\mathcal{D}(G)$ and $\mathcal{D}(D),$ respectively. We consider
the following operator $A:\mathcal{D}(A)\subseteq H_{\nu,0}(\mathbb{R};H_{0}\oplus H_{1})\to H_{\nu,0}(\mathbb{R};H_{0}\oplus H_{1})$
given by 
\begin{align}
\mathcal{D}(A) & \coloneqq\left\{ (u,v)\in H_{\nu,0}\left(\mathbb{R};\mathcal{D}(G)\oplus\mathcal{D}(D)\right)\left|\left(\pi_{\mathcal{BD}(G)}u,\Bullet D\pi_{\mathcal{BD}(D)}v\right)\in h\right.\right\} ,\label{eq:A}\\
A\left(\begin{array}{c}
u\\
v
\end{array}\right) & \coloneqq\left(\begin{array}{cc}
0 & D\\
G & 0
\end{array}\right)\left(\begin{array}{c}
u\\
v
\end{array}\right).\nonumber 
\end{align}

Here $h\subseteq H_{\nu,0}\left(\mathbb{R};\mathcal{BD}(G)\right)\oplus H_{\nu,0}\left(\mathbb{R};\mathcal{BD}(G)\right)$
is a binary relation. We will show that under appropriate assumptions
on $h,$ the operator $A$ satisfies the hypotheses (H2) and (H3)
and thus \prettyref{thm:well_posedness} and \prettyref{prop:causality}
are applicable. The first subsection is devoted to the proof of the
following theorem.
\begin{thm}
\label{thm:A_hypotheses}If $h$ satisfies the hypotheses (H2) or
(H3), then so does $A.$ 
\end{thm}
In Subsection 4.2 we will illustrate the applicability of the abstract
class of evolutionary inclusions of the form \prettyref{eq:evol},
where the relation $A$ is an operator of the form given in \prettyref{eq:A}.

\subsection{The proof of Theorem 4.1.}

We begin to show the monotonicity of $A$ under the assumption that
$h$ is monotone. Indeed, we will show an equality, which will imply,
despite the monotonicity of $A$, that $A$ satisfies (H3) if $h$
does.
\selectlanguage{english}%
\begin{lem}
\label{lem:monotone}\foreignlanguage{american}{Let $(u,v),(x,y)\in\mathcal{D}(A).$
Then for each $a\in\mathbb{R}$ the following equality holds: 
\begin{align*}
\Re\intop_{-\infty}^{a}\left\langle \left.A\left(\begin{array}{c}
u\\
v
\end{array}\right)(t)-A\left(\begin{array}{c}
x\\
y
\end{array}\right)(t)\right|\left(\begin{array}{c}
u\\
v
\end{array}\right)(t)-\left(\begin{array}{c}
x\\
y
\end{array}\right)(t)\right\rangle _{H_{0}\oplus H_{1}}e^{-2\nu t}\,\mathrm{d}t\\
=\Re\intop_{-\infty}^{a}\left\langle \pi_{\mathcal{BD}(G)}\left(u(t)-x(t)\right)\left|\Bullet D\pi_{\mathcal{BD}(D)}\left(v(t)-y(t)\right)\right.\right\rangle _{\mathcal{BD}(G)}e^{-2\nu t}\,\mathrm{d}t.
\end{align*}
In particular, if $h$ is monotone, then so is $A$.}\end{lem}
\selectlanguage{american}%
\begin{proof}
Let $a\in\mathbb{R}.$ Then we compute 
\begin{align*}
 & \Re\intop_{-\infty}^{a}\left\langle \left.A\left(\begin{array}{c}
u\\
v
\end{array}\right)(t)-A\left(\begin{array}{c}
x\\
y
\end{array}\right)(t)\right|\left(\begin{array}{c}
u\\
v
\end{array}\right)(t)-\left(\begin{array}{c}
x\\
y
\end{array}\right)(t)\right\rangle _{H_{0}\oplus H_{1}}e^{-2\nu t}\mbox{ d}t\\
 & =\Re\intop_{-\infty}^{a}\langle D\left(v(t)-y(t)\right)|u(t)-x(t)\rangle_{H_{0}}e^{-2\nu t}\mbox{ d}t+\Re\intop_{-\infty}^{a}\langle G\left(u(t)-x(t)\right)|v(t)-y(t)\rangle_{H_{1}}e^{-2\nu t}\mbox{ d}t\\
 & =\Re\intop_{-\infty}^{a}\langle D\left(v(t)-y(t)\right)|u(t)-x(t)\rangle_{H_{0}}e^{-2\nu t}\mbox{ d}t\\
 & \quad+\Re\intop_{-\infty}^{a}\langle GP_{\mathcal{BD}(G)}\left(u(t)-x(t)\right)|v(t)-y(t)\rangle_{H_{1}}e^{-2\nu t}\mbox{ d}t\\
 & \quad+\Re\intop_{-\infty}^{a}\left\langle \left.G_{c}\left(1-P_{\mathcal{BD}(G)}\right)\left(u(t)-x(t)\right)\right|v(t)-y(t)\right\rangle _{H_{1}}e^{-2\nu t}\mbox{ d}t\\
 & =\Re\intop_{-\infty}^{a}\langle D\left(v(t)-y(t)\right)|u(t)-x(t)\rangle_{H_{0}}e^{-2\nu t}\mbox{ d}t\\
 & \quad+\Re\intop_{-\infty}^{a}\langle GP_{\mathcal{BD}(G)}\left(u(t)-x(t)\right)|v(t)-y(t)\rangle_{H_{1}}e^{-2\nu t}\mbox{ d}t\\
 & \quad-\Re\intop_{-\infty}^{a}\left\langle \left.\left(1-P_{\mathcal{BD}(G)}\right)(u(t)-x(t))\right|D\left(v(t)-y(t)\right)\right\rangle _{H_{0}}e^{-2\nu t}\mbox{ d}t\\
 & =\Re\intop_{-\infty}^{a}\langle GP_{\mathcal{BD}(G)}\left(u(t)-x(t)\right)|v(t)-y(t)\rangle_{H_{1}}e^{-2\nu t}\mbox{ d}t\\
 & \quad+\Re\intop_{-\infty}^{a}\langle P_{\mathcal{BD}(G)}\left(u(t)-x(t)\right)|D\left(v(t)-y(t)\right)\rangle_{H_{0}}e^{-2\nu t}\mbox{ d}t\\
 & =\Re\intop_{-\infty}^{a}\langle GP_{\mathcal{BD}(G)}\left(u(t)-x(t)\right)|P_{\mathcal{BD}(D)}(v(t)-y(t))\rangle_{H_{1}}e^{-2\nu t}\mbox{ d}t\\
 & \quad+\Re\intop_{-\infty}^{a}\left\langle GP_{\mathcal{BD}(G)}\left(u(t)-x(t)\right)\left|\left(1-P_{\mathcal{BD}(D)}\right)\left(v(t)-y(t)\right)\right.\right\rangle _{H_{1}}e^{-2\nu t}\mbox{ d}t\\
 & \quad+\Re\intop_{-\infty}^{a}\langle P_{\mathcal{BD}(G)}\left(u(t)-x(t)\right)|DP_{\mathcal{BD}(D)}\left(v(t)-y(t)\right)\rangle_{H_{0}}e^{-2\nu t}\mbox{ d}t\\
 & \quad+\Re\intop_{-\infty}^{a}\left\langle P_{\mathcal{BD}(G)}\left(u(t)-x(t)\right)\left|D_{c}\left(1-P_{\mathcal{BD}(D)}\right)\left(v(t)-y(t)\right)\right.\right\rangle _{H_{0}}e^{-2\nu t}\mbox{ d}t\\
 & =\Re\intop_{-\infty}^{a}\langle GP_{\mathcal{BD}(G)}\left(u(t)-x(t)\right)|P_{\mathcal{BD}(D)}(v(t)-y(t))\rangle_{H_{1}}e^{-2\nu t}\mbox{ d}t\\
 & \quad+\Re\intop_{-\infty}^{a}\langle P_{\mathcal{BD}(G)}\left(u(t)-x(t)\right)|DP_{\mathcal{BD}(D)}\left(v(t)-y(t)\right)\rangle_{H_{0}}e^{-2\nu t}\mbox{ d}t\\
 & =\Re\intop_{-\infty}^{a}\langle GP_{\mathcal{BD}(G)}\left(u(t)-x(t)\right)|GDP_{\mathcal{BD}(D)}(v(t)-y(t))\rangle_{H_{1}}e^{-2\nu t}\mbox{ d}t\\
 & \quad+\Re\intop_{-\infty}^{a}\langle P_{\mathcal{BD}(G)}\left(u(t)-x(t)\right)|DP_{\mathcal{BD}(D)}\left(v(t)-y(t)\right)\rangle_{H_{0}}e^{-2\nu t}\mbox{ d}t\\
 & =\Re\intop_{-\infty}^{a}\left\langle P_{\mathcal{BD}(G)}\left(u(t)-x(t)\right)\left|DP_{\mathcal{BD}(D)}\left(v(t)-y(t)\right)\right.\right\rangle _{\mathcal{D}(G)}e^{-2\nu t}\mbox{ d}t\\
 & =\Re\intop_{-\infty}^{a}\left\langle \pi_{\mathcal{BD}(G)}\left(u(t)-x(t)\right)\left|\Bullet D\pi_{\mathcal{BD}(D)}\left(v(t)-y(t)\right)\right.\right\rangle _{\mathcal{BD}(G)}e^{-2\nu t}\mbox{ d}t.\tag*{\qedhere}
\end{align*}
\end{proof}
\selectlanguage{english}%
\begin{lem}
\label{lem:closed}\foreignlanguage{american}{If the relation $h$
is closed, then so is $A$. }\end{lem}
\selectlanguage{american}%
\begin{proof}
Let $\left((u_{n},v_{n})\right)_{n\in\mathbb{N}}$ be a sequence in
$\mathcal{D}(A)$ such that $u_{n}\to u,v_{n}\to v$ as well as $Dv_{n}\to w$
and $Gu_{n}\to x$ in $H_{\nu,0}(\mathbb{R};H_{0})$ and $H_{\nu,0}(\mathbb{R};H_{1})$
respectively. Due to the closedness of $G$ and $D$ we obtain that
$u\in H_{\nu,0}(\mathbb{R};\mathcal{D}(G)),v\in H_{\nu,0}(\mathbb{R};\mathcal{D}(D))$
and $Gu=x,Dv=w.$ Note that this means that $u_{n}\to u$ and $v_{n}\to v$
in $H_{\nu,0}(\mathbb{R};\mathcal{D}(G))$ and $H_{\nu,0}(\mathbb{R};\mathcal{D}(D))$,
respectively. The latter implies $\Bullet D\pi_{\mathcal{BD}(D)}v_{n}\to\Bullet D\pi_{\mathcal{BD}(D)}v$
and $\pi_{\mathcal{BD}(G)}u_{n}\to\pi_{\mathcal{BD}(G)}u$ in $H_{\nu,0}(\mathbb{R};\mathcal{BD}(G)).$
Since $\left(\pi_{\mathcal{BD}(G)}u_{n},\Bullet D\pi_{\mathcal{BD}(D)}v_{n}\right)\in h$
for each $n\in\mathbb{N},$ the closedness of $h$ yields $\left(\pi_{\mathcal{BD}(G)}u,\Bullet D\pi_{\mathcal{BD}(D)}v\right)\in h.$
This shows $(u,v)\in\mathcal{D}(A)$ and hence, $A$ is closed.
\end{proof}
Now we are able to state the main result of this subsection.
\begin{prop}
\label{prop:A_max_mon}Let $h$ be maximal monotone. Then $A$ is
maximal monotone, too.\end{prop}
\begin{proof}
Since every maximal monotone relation is closed (in fact it is even
demi-closed, see Subsection 2.2), the operator $A$ is closed according
to \prettyref{lem:closed}. Moreover, by \prettyref{lem:monotone}
$A$ is monotone. Thus, by \prettyref{thm:Minty}, it suffices to
prove that $(1+A)$ has a dense range. To this end let $f\in H_{\nu,0}(\mathbb{R};\mathcal{D}(G_{c}))$
and $g\in H_{\nu,0}(\mathbb{R};\mathcal{D}(D_{c})).$ We define%
\footnote{Note that $1-G_{c}D$ and $1-DG_{c}$ are boundedly invertible, since
$-G_{c}D$ and $-DG_{c}$ are positive selfadjoint operators.%
}\foreignlanguage{english}{ 
\begin{align*}
\tilde{u} & \coloneqq(1-DG_{c})^{-1}f-D(1-G_{c}D)^{-1}g\in H_{\nu,0}(\mathbb{R};\mathcal{D}(DG_{c})),\\
\tilde{v} & \coloneqq(1-G_{c}D)^{-1}g-G_{c}(1-DG_{c})^{-1}f\in H_{\nu,0}(\mathbb{R};\mathcal{D}(D)),
\end{align*}
}as well as\foreignlanguage{english}{ 
\begin{align*}
u & \coloneqq\pi_{\mathcal{BD}(G)}^{\ast}(1+h)^{-1}\left(-\Bullet D\pi_{\mathcal{BD}(D)}G_{c}\tilde{u}\right)+\tilde{u}\in H_{\nu,0}(\mathbb{R};\mathcal{D}(G)),\\
v & \coloneqq-\pi_{\mathcal{BD}(D)}^{\ast}\Bullet G(1+h)^{-1}\left(-\Bullet D\pi_{\mathcal{BD}(D)}G_{c}\tilde{u}\right)+\tilde{v}\in H_{\nu,0}(\mathbb{R};\mathcal{D}(D)).
\end{align*}
}We show that $(u,v)\in\mathcal{D}(A)$ and that $\left(1+A\right)\left(\begin{array}{c}
u\\
v
\end{array}\right)=\left(\begin{array}{c}
f\\
g
\end{array}\right).$ First, we compute 
\begin{align*}
\tilde{u}+D\tilde{v} & =(1-DG_{c})^{-1}f-D(1-G_{c}D)^{-1}g+D\left((1-G_{c}D)^{-1}g-G_{c}(1-DG_{c})^{-1}f\right)\\
 & =(1-DG_{c})(1-DG_{c})^{-1}f\\
 & =f,
\end{align*}
and 
\begin{align*}
\tilde{v}+G_{c}\tilde{u} & =(1-G_{c}D)^{-1}g-G_{c}(1-DG_{c})^{-1}f+G_{c}\left((1-DG_{c})^{-1}f-D(1-G_{c}D)^{-1}g\right)\\
 & =(1-G_{c}D)(1-G_{c}D)^{-1}g\\
 & =g.
\end{align*}
Moreover, we get that 
\begin{align*}
 & u+Dv\\
 & =\pi_{\mathcal{BD}(G)}^{\ast}(1+h)^{-1}\left(-\Bullet D\pi_{\mathcal{BD}(D)}G_{c}\tilde{u}\right)+\tilde{u}+D\left(-\pi_{\mathcal{BD}(D)}^{\ast}\Bullet G(1+h)^{-1}\left(-\Bullet D\pi_{\mathcal{BD}(D)}G_{c}\tilde{u}\right)+\tilde{v}\right)\\
 & =\pi_{\mathcal{BD}(G)}^{\ast}(1+h)^{-1}\left(-\Bullet D\pi_{\mathcal{BD}(D)}G_{c}\tilde{u}\right)+\tilde{u}-\pi_{\mathcal{BD}(G)}^{\ast}\Bullet D\Bullet G(1+h)^{-1}\left(-\Bullet D\pi_{\mathcal{BD}(D)}G_{c}\tilde{u}\right)+D\tilde{v}\\
 & =\tilde{u}+D\tilde{v}=f,
\end{align*}
as well as 
\begin{align*}
 & v+Gu\\
 & =-\pi_{\mathcal{BD}(D)}^{\ast}\Bullet G(1+h)^{-1}\left(-\Bullet D\pi_{\mathcal{BD}(D)}G_{c}\tilde{u}\right)+\tilde{v}+G\left(\pi_{\mathcal{BD}(G)}^{\ast}(1+h)^{-1}\left(-\Bullet D\pi_{\mathcal{BD}(D)}G_{c}\tilde{u}\right)+\tilde{u}\right)\\
 & =\tilde{v}+G_{c}\tilde{u}=g,
\end{align*}
which shows that $\left(1+\left(\begin{array}{cc}
0 & D\\
G & 0
\end{array}\right)\right)\left(\begin{array}{c}
u\\
v
\end{array}\right)=\left(\begin{array}{c}
f\\
g
\end{array}\right).$ Hence, the only thing which is left to show is $\left(\pi_{\mathcal{BD}(G)}u,\Bullet D\pi_{\mathcal{BD}(D)}v\right)\in h.$
Since $\tilde{u}$ takes values in $\mathcal{D}(G_{c})$ we get that
\begin{equation}
\pi_{\mathcal{BD}(G)}u=(1+h)^{-1}\left(-\Bullet D\pi_{\mathcal{BD}(D)}G_{c}\tilde{u}\right).\label{eq:fixpoint}
\end{equation}
Using that 
\[
\pi_{\mathcal{BD}(D)}G_{c}D(1-G_{c}D)^{-1}g=-\pi_{\mathcal{BD}(D)}g+\pi_{\mathcal{BD}(D)}(1-G_{c}D)^{-1}g=\pi_{\mathcal{BD}(D)}(1-G_{c}D)^{-1}g,
\]
we obtain 
\begin{align*}
\pi_{\mathcal{BD}(D)}G_{c}\tilde{u} & =\pi_{\mathcal{BD}(D)}G_{c}\left((1-DG_{c})^{-1}f-D(1-G_{c}D)^{-1}g\right)\\
 & =\pi_{\mathcal{BD}(D)}G_{c}\left(1-DG_{c}\right)^{-1}f-\pi_{\mathcal{BD}(D)}\left(1-G_{c}D\right)^{-1}g\\
 & =-\pi_{\mathcal{BD}(D)}\tilde{v}.
\end{align*}
Thus, we compute 
\begin{align*}
 & \pi_{\mathcal{BD}(D)}v+\Bullet G\pi_{\mathcal{BD}(G)}u\\
 & =-\Bullet G(1+h)^{-1}\left(-\Bullet D\pi_{\mathcal{BD}(D)}G_{c}\tilde{u}\right)+\pi_{\mathcal{BD}(D)}\tilde{v}+\Bullet G(1+h)^{-1}\left(-\Bullet D\pi_{\mathcal{BD}(D)}G_{c}\tilde{u}\right)\\
 & =\pi_{\mathcal{BD}(D)}\tilde{v}=-\pi_{\mathcal{BD}(D)}G_{c}\tilde{u},
\end{align*}
which gives 
\[
\Bullet D\pi_{\mathcal{BD}(D)}v+\pi_{\mathcal{BD}(G)}u=-\Bullet D\pi_{\mathcal{BD}(D)}G_{c}\tilde{u}.
\]
Together with \prettyref{eq:fixpoint} the latter yields 
\[
\pi_{\mathcal{BD}(G)}u=(1+h)^{-1}\left(\Bullet D\pi_{\mathcal{BD}(D)}v+\pi_{\mathcal{BD}(G)}u\right),
\]
which is equivalent to $(\pi_{\mathcal{BD}(G)}u,\Bullet D\pi_{\mathcal{BD}(D)}v)\in h.$
\end{proof}
Summarizing we have shown that if $h$ satisfies the hypothesis (H3),
then so does $A$, which follows by \prettyref{lem:monotone}. Moreover,
it follows that if $h$ satisfies (H2), then $A$ satisfies (H2) as
well. Indeed, the maximal monotonicity was shown in \prettyref{prop:A_max_mon}
while the statement that $A$ is autonomous if $h$ is autonomous
follows directly by the definition of $A.$ Hence, we have proved
\prettyref{thm:A_hypotheses}.

\subsection{Applications}

In this subsection we apply our findings of the Sections 3 and 4.1
to two examples. The first problem is linear and deals with the model
of acoustic waves with an impedance-type boundary condition (see \cite{Picard2012_Impedance}).
The second one, describing a frictional contact problem for a visco-elastic
medium, was also treated by Migorski et al. (see \cite{Migorski_2009})
for an antiplane model, using a variational approach.

\subsubsection*{Wave propagation with impedance-type boundary conditions}

Throughout let $\Omega\subseteq\mathbb{R}^{n}$. We define the following
differential operators, which will play the role of $G$ and $D$. 
\begin{defn*}
We define $\grad_{c}$ as the closure of the operator 
\begin{align*}
\grad|_{C_{c}^{\infty}(\Omega)}:C_{c}^{\infty}(\Omega)\subseteq L_{2}(\Omega) & \to L_{2}(\Omega)^{n}\\
\phi & \mapsto\left(\partial_{i}\phi\right)_{i\in\{1,\ldots,n\}}
\end{align*}
and likewise $\dive_{c}$ as the closure of 
\begin{align*}
\dive|_{C_{c}^{\infty}(\Omega)^{n}}:C_{c}^{\infty}(\Omega)^{n}\subseteq L_{2}(\Omega)^{n} & \to L_{2}(\Omega)\\
\left(\psi_{i}\right)_{i\in\{1,\ldots,n\}} & \mapsto\sum_{i=1}^{n}\partial_{i}\psi_{i}.
\end{align*}
\end{defn*}
\begin{rem}
The operators $\grad_{c}$ and $\dive_{c}$ are formally skew-adjoint
and as above we set $\grad\coloneqq-\left(\dive_{c}\right)^{\ast}$
and $\dive\coloneqq-\left(\grad_{c}\right)^{\ast}.$ Note that an
element $u$ lying in the domain of $\grad_{c}$ satisfies the abstract
Dirichlet-boundary condition ``$u=0$'' on $\partial\Omega.$ Likewise,
an element $q\in\mathcal{D}\left(\dive_{c}\right)$ satisfies the
Neumann-boundary condition ``$q\cdot N=0$'' on $\partial\Omega$
in case of a smooth boundary $\partial\Omega,$ where $N$ denotes
the unit outer normal vector-field on $\partial\Omega$. The corresponding
boundary data spaces are given by $\mathcal{BD}(\grad)=[\{0\}](1-\dive\grad),$
the space of 1-harmonic functions on $\Omega$, and $\mathcal{BD}(\dive)=[\{0\}](1-\grad\dive).$ 
\end{rem}
Following \cite{Picard2012_Impedance} we consider the following boundary
value problem:
\begin{align}
\partial_{0,\nu}^{2}u-\dive\grad u & =f\mbox{ on }\Omega,\label{eq:wave}\\
\left(\partial_{0,\nu}^{2}a(\partial_{0,\nu}^{-1})u+\grad u\right)\cdot N & =0\mbox{ on }\partial\Omega.\label{eq:wave_bd}
\end{align}
Here $a:B_{\mathbb{C}}(r,r)\to L\left(L_{2}(\Omega)^{n}\right)$ is
a linear material law for some $r>0$, having the power series representation
\begin{equation}
a(z)=\sum_{k=0}^{\infty}a_{k}(m)(z-r)^{k},\label{eq:power-series_a}
\end{equation}
where the $a_{k}$'s are $L_{\infty}(\Omega)$-vector fields, and
$a_{k}(m)$ denotes the corresponding multiplication operator. Moreover
we assume that $\dive a_{k}\in L_{\infty}(\Omega)$ such that 
\begin{equation}
\left(\dive a\right):B_{\mathbb{C}}(r,r)\ni z\mapsto\sum_{k=0}^{\infty}\left(\dive a_{k}\right)(m)(z-r)^{k}\in L\left(L_{2}(\Omega)\right)\label{eq:div a}
\end{equation}
gets bounded. We choose $\nu>\frac{1}{2r}.$ Using the product rule
$\dive\left(a(\partial_{0,\nu}^{-1})p\right)=\left(\left(\dive a\right)(\partial_{0,\nu}^{-1})\right)p+a(\partial_{0,\nu}^{-1})\grad p$
for $p\in\mathcal{D}(\grad),$ we can establish $a(\partial_{0,\nu}^{-1})$
as a bounded operator 
\[
a(\partial_{0,\nu}^{-1}):H_{\nu,0}(\mathbb{R};\mathcal{D}(\grad))\to H_{\nu,0}(\mathbb{R};\mathcal{D}(\dive)).
\]
Moreover, one can show that $a\left(\partial_{0,\nu}^{-1}\right)\left[H_{\nu,0}\left(\mathbb{R};\mathcal{D}(\grad_{c})\right)\right]\subseteq H_{\nu,0}\left(\mathbb{R};\mathcal{D}(\dive_{c})\right).$
We define $h:\mathcal{D}(h)\subseteq H_{\nu,0}(\mathbb{R};\mathcal{BD}(\grad))\to H_{\nu,0}(\mathbb{R};\mathcal{BD}(\grad))$
by $h\coloneqq\partial_{0,\nu}\Bullet\dive\pi_{\mathcal{BD}(\dive)}a(\partial_{0,\nu}^{-1})\pi_{\mathcal{BD}(\grad)}^{\ast}$
with maximal domain, which is a densely defined closed linear operator.
Consider the operator 
\begin{align*}
A:\mathcal{D}(A)\subseteq H_{\nu,0}\left(\mathbb{R};L_{2}(\Omega)\oplus L_{2}(\Omega)^{n}\right) & \to H_{\nu,0}\left(\mathbb{R};L_{2}(\Omega)\oplus L_{2}(\Omega)^{n}\right)\\
\left(\begin{array}{c}
u\\
q
\end{array}\right) & \mapsto\left(\begin{array}{cc}
0 & \dive\\
\grad & 0
\end{array}\right)\left(\begin{array}{c}
u\\
q
\end{array}\right),
\end{align*}
where 
\[
\mathcal{D}(A)\coloneqq\left\{ (u,q)\in H_{\nu,0}\left(\mathbb{R};\mathcal{D}(\grad)\oplus\mathcal{D}(\dive)\right)\,\left|\, h\left(\pi_{\mathcal{BD}(\grad)}u\right)=\Bullet\dive\pi_{\mathcal{BD}(\dive)}q\right.\right\} .
\]
Note that, by definition, the operator $h$ is autonomous in the sense
of (H2). The evolutionary equation for the boundary value problem
\prettyref{eq:wave}, \prettyref{eq:wave_bd} is then given by 
\begin{equation}
\left(\partial_{0,\nu}+A\right)\left(\begin{array}{c}
u\\
q
\end{array}\right)=\left(\begin{array}{c}
\partial_{0,\nu}^{-1}f\\
0
\end{array}\right).\label{eq:wave_evo}
\end{equation}
Indeed, if $\left(\begin{array}{c}
u\\
q
\end{array}\right)$ is a solution of \prettyref{eq:wave_evo} then we formally get that
$\partial_{0,\nu}q+\grad u=0$, which gives $q=-\partial_{0,\nu}^{-1}\grad u.$
Moreover, the first line of \prettyref{eq:wave_evo} gives $\partial_{0,\nu}u+\dive q=\partial_{0,\nu}^{-1}f.$
Hence, the formula for $q$ yields $\partial_{0,\nu}u-\partial_{0,\nu}^{-1}\dive\grad u=\partial_{0,\nu}^{-1}f,$
which gives \prettyref{eq:wave}. Moreover, by the domain of $A$
we get that 
\begin{align*}
\Bullet\dive\pi_{\mathcal{BD}(\dive)}q & =h\left(\pi_{\mathcal{BD}(\grad)}u\right)\\
 & =\partial_{0,\nu}\Bullet\dive\pi_{\mathcal{BD}(\dive)}a(\partial_{0,\nu}^{-1})\pi_{\mathcal{BD}(\grad)}^{\ast}\pi_{\mathcal{BD}(\grad)}u\\
 & =\partial_{0,\nu}\Bullet\dive\pi_{\mathcal{BD}(\dive)}a(\partial_{0,\nu}^{-1})u,
\end{align*}
where in the third equality we have used that $a\left(\partial_{0,\nu}^{-1}\right)\left[H_{\nu,0}\left(\mathbb{R};\mathcal{D}(\grad_{c})\right)\right]\subseteq H_{\nu,0}\left(\mathbb{R};\mathcal{D}(\dive_{c})\right).$
Since $q=-\partial_{0,\nu}^{-1}\grad u$ the latter equation gives
\[
-\partial_{0,\nu}^{-2}\pi_{\mathcal{BD}(\dive)}\grad u=\pi_{\mathcal{BD}(\dive)}a(\partial_{0,\nu}^{-1})u,
\]
which is the appropriate formulation for the boundary condition \prettyref{eq:wave_bd}
within our framework. According to our findings in Subsection 4.1
it suffices to check, whether $h$ satisfies the hypotheses (H2) and
(H3). In \cite{Picard2012_Impedance} we find an additional constraint
on $a(\partial_{0,\nu}^{-1})$, namely 
\begin{equation}
\Re\intop_{-\infty}^{0}\left(\langle\grad u|\partial_{0,\nu}a(\partial_{0,\nu}^{-1})u\rangle(t)+\langle u|\dive\partial_{0,\nu}a(\partial_{0,\nu}^{-1})u\rangle(t)\right)e^{-2\nu t}\mbox{ d}t\geq0\label{eq:cond_Picard}
\end{equation}
for all $u\in H_{\nu,0}(\mathbb{R};\mathcal{D}(\grad))\cap\mathcal{D}(\partial_{0,\nu})$.
This condition implies the hypothesis (H3). Indeed, for $u\in\mathcal{D}(\partial_{0,\nu})\cap H_{\nu,0}(\mathbb{R};\mathcal{BD}(\grad))$
we estimate 
\begin{align*}
 & \Re\intop_{-\infty}^{0}\langle hu|u\rangle_{\mathcal{BD}(\grad)}(t)e^{-2\nu t}\mbox{ d}t\\
 & =\Re\intop_{-\infty}^{0}\left\langle \left.\partial_{0,\nu}\dive P_{\mathcal{BD}(\dive)}a(\partial_{0,\nu}^{-1})\pi_{\mathcal{BD}(\grad)}^{\ast}u\right|\pi_{\mathcal{BD}(\grad)}^{\ast}u\right\rangle _{\mathcal{D}(\grad)}(t)\, e^{-2\nu t}\mbox{ d}t\\
 & =\Re\intop_{-\infty}^{0}\left\langle P_{\mathcal{BD}(\dive)}\partial_{0,\nu}a(\partial_{0,\nu}^{-1})\pi_{\mathcal{BD}(\grad)}^{\ast}u\left|\grad\pi_{\mathcal{BD}(\grad)}^{\ast}u\right.\right\rangle (t)\, e^{-2\nu t}\mbox{ d}t\\
 & \quad+\Re\intop_{-\infty}^{0}\left\langle \left.\dive P_{\mathcal{BD}(\dive)}\partial_{0,\nu}a(\partial_{0,\nu}^{-1})\pi_{\mathcal{BD}(\grad)}^{\ast}u\right|\pi_{\mathcal{BD}(\grad)}^{\ast}u\right\rangle (t)\, e^{-2\nu t}\mbox{ d}t\\
 & =\Re\intop_{-\infty}^{0}\left\langle \partial_{0,\nu}a(\partial_{0,\nu}^{-1})\pi_{\mathcal{BD}(\grad)}^{\ast}u\left|\grad\pi_{\mathcal{BD}(\grad)}^{\ast}u\right.\right\rangle (t)\, e^{-2\nu t}\mbox{ d}t\\
 & \quad-\Re\intop_{-\infty}^{0}\left\langle P_{\mathcal{D}(\dive_{c})}\partial_{0,\nu}a(\partial_{0,\nu}^{-1})\pi_{\mathcal{BD}(\grad)}^{\ast}u\left|\grad\pi_{\mathcal{BD}(\grad)}^{\ast}u\right.\right\rangle (t)\, e^{-2\nu t}\mbox{ d}t\\
 & \quad+\Re\intop_{-\infty}^{0}\left\langle \left.\dive\partial_{0,\nu}a(\partial_{0,\nu}^{-1})\pi_{\mathcal{BD}(\grad)}^{\ast}u\right|\pi_{\mathcal{BD}(\grad)}^{\ast}u\right\rangle (t)\, e^{-2\nu t}\mbox{ d}t\\
 & \quad-\Re\intop_{-\infty}^{0}\left\langle \left.\dive_{c}P_{\mathcal{D}(\dive_{c})}\partial_{0,\nu}a(\partial_{0,\nu}^{-1})\pi_{\mathcal{BD}(\grad)}^{\ast}u\right|\pi_{\mathcal{BD}(\grad)}^{\ast}u\right\rangle (t)\, e^{-2\nu t}\mbox{ d}t\\
 & =\Re\intop_{-\infty}^{0}\left\langle \partial_{0,\nu}a(\partial_{0,\nu}^{-1})\pi_{\mathcal{BD}(\grad)}^{\ast}u\left|\grad\pi_{\mathcal{BD}(\grad)}^{\ast}u\right.\right\rangle (t)\, e^{-2\nu t}\mbox{ d}t\\
 & \quad+\Re\intop_{-\infty}^{0}\left\langle \dive\partial_{0,\nu}a(\partial_{0,\nu}^{-1})\pi_{\mathcal{BD}(\grad)}^{\ast}u\left|\pi_{\mathcal{BD}(\grad)}^{\ast}u\right.\right\rangle (t)\, e^{-2\nu t}\mbox{ d}t\\
 & \geq0.
\end{align*}
Since $\mathcal{D}(\partial_{0,\nu})\cap H_{\nu,0}(\mathbb{R};\mathcal{BD}(\grad))$
is a core of $h$, according to \prettyref{lem:core} and since $h$
commutes with the operators $\tau_{h}$ for each $h\in\mathbb{R},$
this yields the monotonicity of $h.$ For showing the maximal monotonicity
of $h,$ we determine its adjoint. Using \prettyref{lem:core} we
get that 
\[
h^{\ast}=\partial_{0,\nu}^{\ast}\pi_{\mathcal{BD}(\grad)}a(\partial_{0,\nu}^{-1})^{\ast}\pi_{\mathcal{BD}(\dive)}^{\ast}\Bullet\grad,
\]
and we obtain that $\mathcal{D}(\partial_{0,\nu})\cap H_{\nu,0}(\mathbb{R};\mathcal{BD}(\grad))=\mathcal{D}(\partial_{0,\nu}^{\ast})\cap H_{\nu,0}(\mathbb{R};\mathcal{BD}(\grad))$
is a core of $h^{\ast}.$ Hence, $h^{\ast}$ is monotone, since for
each $u\in\mathcal{D}(\partial_{0,\nu})\cap H_{\nu,0}(\mathbb{R};\mathcal{BD}(\grad))\subseteq\mathcal{D}(h)$
we estimate 
\[
\Re\langle h^{\ast}u|u\rangle_{H_{\nu,0}(\mathbb{R};\mathcal{BD}(\grad))}=\Re\langle u|hu\rangle_{H_{\nu,0}(\mathbb{R};\mathcal{BD}(\grad))}\geq0.
\]
We summarize our findings in the following theorem.
\begin{thm}
Let $a:B_{\mathbb{C}}(r,r)\to L(L_{2}(\Omega)^{n})$ be of the form
\prettyref{eq:power-series_a} and let $a$ satisfy \prettyref{eq:div a}
and \prettyref{eq:cond_Picard}. Then the evolutionary equation \prettyref{eq:wave_evo}
is well-posed and the corresponding solution operator is causal.\end{thm}
\begin{proof}
Since $h$ and $h^{\ast}$ are monotone, we get the maximal monotonicity
of $h,$ since $h$ is closed. Thus, $h$ satisfies the hypotheses
(H2) and (H3). \prettyref{thm:A_hypotheses} and \prettyref{thm:well_posedness}
now give the well-posedness of \prettyref{eq:wave_evo} and \prettyref{prop:causality}
shows the causality of the solution operator.
\end{proof}

\subsubsection*{Visco-elasticity with frictional boundary conditions}

This subsection is devoted to the study of the following system:
\begin{align}
\rho\partial_{0,\nu}^{2}u-\Div T & =f,\label{eq:visco_elastic}\\
T & =C\Grad u+D\Grad\partial_{0,\nu}u\label{eq:Kelvin-Voigt}
\end{align}
on a domain $\Omega\subseteq\mathbb{R}^{n}$ completed by a nonlinear
boundary condition of the form 
\begin{equation}
(\partial_{0,\nu}u,-T\cdot N)\in g,\label{eq:frictional boundary}
\end{equation}
where $g\subseteq H_{\nu,0}(\mathbb{R};L_{2}(\partial\Omega))\oplus H_{\nu,0}(\mathbb{R};L_{2}(\partial\Omega))$
and $N$ denotes the outward unit normal vector-field on $\partial\Omega.$
Here $u\in H_{\nu,0}(\mathbb{R};L_{2}(\Omega)^{n}),$ denoting the
displacement-field of the body $\Omega$, and $T\in H_{\nu,0}(\mathbb{R};H_{\mathrm{sym}}(\Omega)),$
standing for the stress tensor, are the unknowns. Equation \prettyref{eq:visco_elastic}
is the well-known equation of elasticity, describing the deformation
of an elastic body $\Omega$ with a density distribution $\rho\in L_{\infty}(\Omega)$,
which is assumed to be real-valued and strictly positive. The operator
$\Div$ denotes the divergence for symmetric-matrix-valued functions
and the function $f\in H_{\nu,0}(\mathbb{R};L_{2}(\Omega)^{n})$ is
an external force. The constitutive relation \prettyref{eq:Kelvin-Voigt},
linking the stress $T$ and the strain $\Grad u$, where by $\Grad$
we denote the symmetrized Jacobian of a vector-valued function, is
known as the Kelvin-Voigt model in visco-elasticity. Here $C$ and
$D$ denote the elasticity and the viscosity tensor, respectively.
The boundary condition \prettyref{eq:frictional boundary} is motivated
by so-called frictional boundary conditions, where $g$ is usually
a subgradient of a convex, lower semicontinuous function (due to Rockafellar
and Tyrell \cite{Rockafellar_1970_book}) or a generalized gradient
of a locally Lipschitz-continuous function (due to Clarke \cite{clarke1976generalized}).
For typical examples of frictional boundary conditions in elasticity
we refer to \cite[p. 171 ff.]{Sofonea2009} and to \cite[Section 5]{Migorski_2009}.\\
A similar problem was treated in \cite{Migorski_2009} for antiplane
shear deformations, where the domain $\Omega$ was assumed to be a
cylinder in $\mathbb{R}^{3}$. In this case the displacement can be
expressed by a scalar-valued function and hence, \prettyref{eq:visco_elastic}
reduces to an equation similar to \prettyref{eq:wave}. The boundary
condition \prettyref{eq:frictional boundary} was assumed to hold
on a part of the boundary, while on the other parts of $\partial\Omega$
Dirichlet- and Neumann-boundary conditions were prescribed. The relation
$g$ was assumed to be a generalized gradient of a locally Lipschitz-continuous
function and the existence of a solution was proved, using a variational
approach. The uniqueness of a solution was shown under the additional
constraint that $g$ is quasi-monotone (i.e. there exists some constant
$c>0$ such that $g+c$ is monotone). Here, for simplicity we assume
that \prettyref{eq:frictional boundary} holds on the whole boundary
and that $g$ satisfies a suitable monotonicity constraint. However,
since our relation is allowed to act in time and space, it covers
a broader class of possible boundary conditions.\\
We begin by introducing the differential operators $\Grad$ and $\Div$.
\begin{defn*}
Let $\Omega\subseteq\mathbb{R}^{n}$ be open and denote by $L_{2}(\Omega)^{n\times n}$
the space of $n\times n$-matrix-valued $L_{2}(\Omega)$-functions
equipped with the Frobenius inner product given by 
\[
\langle\Phi|\Psi\rangle\coloneqq\intop_{\Omega}\trace(\Phi(t)^{\ast}\Psi(t))\mbox{ d}t\quad\left(\Phi,\Psi\in L_{2}(\Omega)^{n\times n}\right).
\]
Moreover, we set $H_{\mathrm{sym}}(\Omega)\coloneqq\left\{ \left.\Phi\in L_{2}(\Omega)^{n\times n}\,\right|\,\Phi(x)^{T}=\Phi(x)\quad(x\in\mathbb{R}\mbox{ a.e.})\right\} ,$
which is a closed subspace of $L_{2}(\Omega)^{n\times n}$ and hence,
a Hilbert space. We define the operator $\grad_{c}$ as the closure
of 
\begin{align*}
\grad|_{C_{c}^{\infty}(\Omega)^{n}}:C_{c}^{\infty}(\Omega)^{n}\subseteq L_{2}(\Omega)^{n} & \to L_{2}(\Omega)^{n\times n}\\
\left(\phi_{i}\right)_{i\in\{1,\ldots,n\}} & \mapsto\left(\partial_{j}\phi_{i}\right)_{i,j\in\{1,\ldots,n\}}
\end{align*}
and $\dive_{c}$ as the closure of 
\begin{align*}
\dive|_{C_{c}^{\infty}(\Omega)^{n\times n}}:C_{c}^{\infty}(\Omega)^{n\times n}\subseteq L_{2}(\Omega)^{n\times n} & \to L_{2}(\Omega)^{n}\\
\left(\Psi_{ij}\right)_{i,j\in\{1,\ldots,n\}} & \mapsto\left(\sum_{j=1}^{n}\partial_{j}\Psi_{ij}\right)_{i\in\{1,\ldots,n\}}.
\end{align*}
Furthermore, we introduce the operators $\Grad_{c}$ and $\Div_{c}$
as the closures of 
\begin{align*}
\Grad|_{C_{c}^{\infty}(\Omega)^{n}}:C_{c}^{\infty}(\Omega)^{n}\subseteq L_{2}(\Omega)^{n} & \to H_{\mathrm{sym}}(\Omega)\\
\left(\phi_{i}\right)_{i\in\{1,\ldots,n\}} & \mapsto\left(\frac{1}{2}\left(\partial_{i}\phi_{j}+\partial_{j}\phi_{i}\right)\right)_{i,j\in\{1,\ldots,n\}}
\end{align*}
and of 
\begin{align*}
\Div|_{C_{c}^{\infty}(\Omega)^{n\times n}\cap H_{\mathrm{sym}}(\Omega)}:C_{c}^{\infty}(\Omega)^{n\times n}\cap H_{\mathrm{sym}}(\Omega)\subseteq H_{\mathrm{sym}}(\Omega) & \to L_{2}(\Omega)^{n}\\
\left(\Psi_{ij}\right)_{i,j\in\{1,\ldots,n\}} & \mapsto\left(\sum_{j=1}^{n}\partial_{j}\Psi_{ij}\right)_{i\in\{1,\ldots,n\}},
\end{align*}
respectively%
\footnote{By definition, elements of $\mathcal{D}(\Grad_{c})$ and of $\mathcal{D}(\Div_{c})$
satisfy a generalized boundary condition of the form $\phi=0$ on
$\partial\Omega$ and $\Psi\cdot N=0$ on $\partial\Omega$ for $\phi\in\mathcal{D}(\Grad_{c})$
and $\Psi\in\mathcal{D}(\Div_{c})$, respectively, where $N$ denotes
the outward normal vector-field on $\partial\Omega.$ %
}.\end{defn*}
\begin{rem}
An easy computation shows that $\grad_{c}$ and $\dive_{c}$ as well
as $\Grad_{c}$ and $\Div_{c}$ are formally skew-adjoint and we define
$\grad\coloneqq\left(-\dive_{c}\right)^{\ast}\supseteq\grad_{c}$,
$\dive\coloneqq\left(-\grad_{c}\right)^{\ast}\supseteq\dive_{c}$
as well as $\Grad\coloneqq-\left(\Div_{c}\right)^{\ast}\supseteq\Grad_{c}$
and $\Div\coloneqq-\left(\Grad_{c}\right)^{\ast}\supseteq\Div_{c}.$
\end{rem}
Throughout we may assume that $\mathcal{D}(\Grad)\stackrel{\iota}{\hookrightarrow}\mathcal{D}(\grad),$
i.e. Korn's inequality holds (for sufficient conditions for Korn's
inequality we refer to \cite{Ciarlet2005} and the references therein).
We rewrite the equations \prettyref{eq:visco_elastic} and \prettyref{eq:Kelvin-Voigt}
into a system of the form \prettyref{eq:evol}. We assume that $D\in L(H_{\mathrm{sym}}(\Omega))$
is a selfadjoint, strictly positive definite operator and hence, we
may rewrite \prettyref{eq:Kelvin-Voigt} as follows 
\[
T=\left(C+D\partial_{0,\nu}\right)\Grad u=\partial_{0,\nu}D(\partial_{0,\nu}^{-1}D^{-1}C+1)\Grad u,
\]
which gives 
\[
\partial_{0,\nu}^{-1}D^{-1}T=(\partial_{0,\nu}^{-1}D^{-1}C+1)\Grad u.
\]
If we choose $\nu>0$ large enough, such that $\|D^{-1}C\|<\nu$,
the latter equation can be written as 
\[
\partial_{0,\nu}^{-1}\left(\partial_{0,\nu}^{-1}D^{-1}C+1\right)^{-1}D^{-1}T=\Grad u.
\]
Thus, the equations \prettyref{eq:visco_elastic} and \prettyref{eq:Kelvin-Voigt}
can be replaced by the system 
\[
\left(\partial_{0,\nu}\left(\begin{array}{cc}
\rho & 0\\
0 & \partial_{0,\nu}^{-1}\left(\partial_{0,\nu}^{-1}D^{-1}C+1\right)^{-1}D^{-1}
\end{array}\right)+\left(\begin{array}{cc}
0 & \Div\\
\Grad & 0
\end{array}\right)\right)\left(\begin{array}{c}
v\\
-T
\end{array}\right)=\left(\begin{array}{c}
f\\
0
\end{array}\right),
\]
where $v\coloneqq\partial_{0,\nu}u.$ Hence, the material law $M(\partial_{0,\nu}^{-1})$
is given by 
\begin{equation}
M(\partial_{0,\nu}^{-1})=\left(\begin{array}{cc}
\rho & 0\\
0 & 0
\end{array}\right)+\partial_{0,\nu}^{-1}\left(\begin{array}{cc}
0 & 0\\
0 & D^{-1}
\end{array}\right)-\partial_{0,\nu}^{-2}\left(\begin{array}{cc}
0 & 0\\
0 & \left(\partial_{0,\nu}^{-1}C+D\right)^{-1}CD^{-1}
\end{array}\right).\label{eq:matrial_law_elasticity}
\end{equation}
According to the assumptions on $\rho$ and $D,$ this material law
satisfies the hypothesis (H1) if we choose $\nu>0$ large enough (this
corresponds to the so-called $0$-analytic case in \cite{Picard}).\\
We now have to formulate the boundary condition \prettyref{eq:frictional boundary}
in our framework. Since we do not want to impose regularity assumptions
on the boundary $\partial\Omega,$ we have to detour the space $L_{2}(\partial\Omega).$
Following \cite[p. 16 ff.]{Picard2012_boundary_control} we define
the following substitute for $L_{2}(\partial\Omega).$
\begin{defn*}
We assume that $N\in L_{\infty}(\Omega)^{n}$ with $\dive N\in L_{\infty}(\Omega).$%
\footnote{This could be interpreted as a ``smooth'' continuation of the outward
normal vector field $N$ to the whole domain $\Omega$. %
} We define the operator $\nu:\mathcal{BD}(\grad)\to\mathcal{BD}(\dive)$
by $\nu f\coloneqq\pi_{\mathcal{BD}(\dive)}\left(f_{j}N_{k}\right)_{j,k\in\{1,\ldots,n\}}$
and assume that 
\[
\left\langle \left.\left(\Bullet\dive\nu+\nu^{\ast}\Bullet\grad\right)f\right|f\right\rangle _{\mathcal{BD}(\grad)}>0\quad(f\in\mathcal{BD}(\grad)\setminus\{0\}).
\]
This gives that $(f,g)\mapsto\frac{1}{2}\left(\left\langle \nu f\left|\Bullet\grad g\right.\right\rangle _{\mathcal{BD}(\dive)}+\left\langle \left.\Bullet\grad f\right|\nu g\right\rangle _{\mathcal{BD}(\dive)}\right)$
defines an inner product on $\mathcal{BD}(\grad)$ and we denote the
completion of $\mathcal{BD}(\grad)$ with respect to this inner product
by $U.$ Moreover, we denote the embedding of $\mathcal{BD}(\grad)$
into $U$ by $\kappa$. \end{defn*}
\begin{rem}
\label{rem: U}A formal calculation (see \cite[Remark 5.2]{Picard2012_boundary_control})
yields that in the case of a smooth boundary we have that 
\[
\langle f|g\rangle_{U}=\intop_{\partial\Omega}f(x)^{\ast}g(x)\mbox{ d}S(x).
\]
We define the mapping $j\coloneqq\kappa\pi_{\mathcal{BD}(\grad)}\iota\pi_{\mathcal{BD}(\Grad)}^{\ast}:\mathcal{BD}(\Grad)\to U.$
Then, as it was pointed out in \cite[Remark 5.4]{Picard2012_boundary_control},
the term $\Bullet\Grad j^{\ast}u$ can be identified with $u$ on
$\partial\Omega.$ This construction allows us to compare Dirichlet-
and Neumann-type boundary values.
\end{rem}
Using this framework, we assume that $g$ is a binary relation on
$H_{\nu,0}(\mathbb{R};U).$ Then the boundary condition \prettyref{eq:frictional boundary}
can be formulated as 
\[
\left(\pi_{\mathcal{BD}(\Grad)}v,\pi_{\mathcal{BD}(\Div)}(-T)\right)\in\Bullet\Grad j^{\ast}gj,
\]
or equivalently as 
\[
\left(\pi_{\mathcal{BD}(\Grad)}v,\Bullet\Div\pi_{\mathcal{BD}(\Div)}(-T)\right)\in j^{\ast}gj.
\]
We define the operator $A$ by 
\begin{align}
\mathcal{D}(A) & \coloneqq\left\{ (v,\sigma)\in H_{\nu,0}(\mathbb{R};\mathcal{D}(\Grad)\oplus\mathcal{D}(\Div))\,\left|\,\left(\pi_{\mathcal{BD}(\Grad)}v,\Bullet\Div\pi_{\mathcal{BD}(\Div)}\sigma\right)\in j^{\ast}gj\right.\right\} \nonumber \\
A\left(\begin{array}{c}
v\\
\sigma
\end{array}\right) & \coloneqq\left(\begin{array}{cc}
0 & \Div\\
\Grad & 0
\end{array}\right)\left(\begin{array}{c}
v\\
\sigma
\end{array}\right),\label{eq:visco-elastic_A}
\end{align}
and consider the equation 
\begin{equation}
\left(\partial_{0,\nu}\left(\begin{array}{cc}
\rho & 0\\
0 & \partial_{0,\nu}^{-1}\left(\partial_{0,\nu}^{-1}D^{-1}C+1\right)^{-1}D^{-1}
\end{array}\right)+A\right)\left(\begin{array}{c}
v\\
\sigma
\end{array}\right)=\left(\begin{array}{c}
f\\
0
\end{array}\right),\label{eq:visco-elastic_evo}
\end{equation}
which is, according to our findings, an appropriate reformulation
of the problem given by \prettyref{eq:visco_elastic}, \prettyref{eq:Kelvin-Voigt}
and \prettyref{eq:frictional boundary}. Hence, we are in the situation
of Subsection 4.1 and if we assume that $j^{\ast}gj\subseteq H_{\nu,0}(\mathbb{R};\mathcal{BD}(\Grad))\oplus H_{\nu,0}(\mathbb{R};\mathcal{BD}(\Grad))$
satisfies the hypotheses (H2) and (H3), \prettyref{thm:A_hypotheses}
applies, which gives the well-posedness and the causality of the problem.
We summarize our findings in the following theorem.
\begin{thm}
Let $\nu>\|D^{-1}C\|$ be such that $M(\partial_{0,\nu}^{-1})$, given
by \prettyref{eq:matrial_law_elasticity}, satisfies (H1) and let
$g\subseteq H_{\nu,0}(\mathbb{R};U)\oplus H_{\nu,0}(\mathbb{R};U)$
be such that $j^{\ast}gj\subseteq H_{\nu,0}(\mathbb{R};\mathcal{BD}(\Grad))\oplus H_{\nu,0}(\mathbb{R};\mathcal{BD}(\Grad))$
satisfies the hypotheses (H2) and (H3). Then the problem given by
\prettyref{eq:visco-elastic_evo}, where $A$ is defined by \prettyref{eq:visco-elastic_A}
is well-posed and the corresponding solution operator is causal. 
\end{thm}
In the last part of this subsection we will present a condition on
$g$, which implies the maximal monotonicity of $j^{\ast}gj.$\\
In \cite{Migorski_2009} the relation $g$ is defined as the canonical
extension (see \prettyref{rem:autonomous relations} (a)) of a quasi-monotone
relation $\tilde{g}$ on $L_{2}(\partial\Omega)$, which is given
by 
\[
\tilde{g}\coloneqq\{(u,v)\in L_{2}(\partial\Omega)\oplus L_{2}(\partial\Omega)\,|\,\left(u(x),v(x)\right)\in\partial j(x,\cdot)\quad(x\in\partial\Omega\mbox{ a.e.})\},
\]
where $j(\cdot,\cdot):\partial\Omega\times\mathbb{R}\to\mathbb{R}$
is a suitable function, which is locally Lipschitz-continuous with
respect to the second variable. Moreover it is assumed that $[\mathbb{R}]\partial j(x,\cdot)=\mathbb{R}$
for almost every $x\in\partial\Omega$ and that there exists $c>0$
such that for every $\xi\in\mathbb{R}$ and every $\eta\in\mathbb{R}$
satisfying $(\xi,\eta)\in\partial j(x,\cdot)$ the estimate $|\eta|\leq c(1+|\xi|)$
holds. This gives that $[L_{2}(\partial\Omega)]\tilde{g}=L_{2}(\partial\Omega)$
and that $\tilde{g}$ is a \emph{bounded} relation, i.e. $\tilde{g}[M]$
is bounded for bounded $M$.\\
A suitable realization of these assumptions in our framework is given
as follows. We assume that $\tilde{g}\subseteq U\oplus U$ is a maximal
monotone bounded relation, satisfying $[U]\tilde{g}=U$ and for technical
reasons $(0,0)\in\tilde{g}.$ In the next lemma we show that these
conditions imply that $j^{\ast}\tilde{g}j$ is maximal monotone, too.
This would give that $j^{\ast}gj$ satisfies the hypotheses (H2) and
(H3), where $g\subseteq H_{\nu,0}(\mathbb{R};U)\oplus H_{\nu,0}(\mathbb{R};U)$
is the canonical extension of $\tilde{g}$. 
\begin{lem}
Let $H_{1},H_{2}$ be two Hilbert spaces, $T\subseteq H_{1}\oplus H_{1}$
maximal monotone and bounded with $(0,0)\in T$ and $[H_{1}]T=H_{1}.$
Moreover, let $S\in L(H_{1},H_{2}).$ Then $STS^{\ast}\subseteq H_{2}\oplus H_{2}$
is maximal monotone.\end{lem}
\begin{proof}
The monotonicity of $STS^{\ast}$ is clear. We will show that $\left(1+STS^{\ast}\right)[H_{2}]=H_{2},$
which implies the maximal monotonicity according to \prettyref{thm:Minty}.
Let $f\in H_{2}.$ We replace the relation $T$ by its Yosida-approximation
$T_{\lambda},$ which is monotone and Lipschitz-continuous with a
Lipschitz-constant less than or equal to $\lambda^{-1}.$ Hence, $ST_{\lambda}S^{\ast}$
is monotone and Lipschitz-continuous with a Lipschitz-constant less
than or equal to $\frac{\|S\|^{2}}{\lambda}.$ For $\mu<\frac{\lambda}{\|S\|^{2}}$
the contraction mapping theorem yields the existence of an element
$x\in H_{2}$ satisfying $x+\mu ST_{\lambda}(S^{\ast}x)=y.$ Hence,
the mapping $ST_{\lambda}S^{\ast}$ is maximal monotone and thus,
for each $\lambda>0$ there exists $x_{\lambda}\in H_{2}$ with $x_{\lambda}+ST_{\lambda}(S^{\ast}x_{\lambda})=y.$
The latter can be written as $x_{\lambda}=J_{1}(ST_{\lambda}S^{\ast})(y),$
which implies that $\left|x_{\lambda}\right|\leq|y|$ for each $\lambda>0,$
according to \prettyref{rem:maximal monotone} (a). Since $T$ is
bounded, this gives that $\{T_{\lambda}(S^{\ast}x_{\lambda})\,|\,\lambda>0\}$
is bounded (see \prettyref{rem:maximal monotone} (b)). We will prove
that $(x_{\lambda})_{\lambda>0}$ converges as $\lambda$ tends to
$0.$ For $\mu,\lambda>0$ we estimate 
\begin{align*}
|x_{\lambda}-x_{\mu}|^{2} & =\Re\langle x_{\lambda}-x_{\mu}|ST_{\mu}(S^{\ast}x_{\mu})-ST_{\lambda}(S^{\ast}x_{\lambda})\rangle\\
 & =\Re\langle S^{\ast}x_{\lambda}-S^{\ast}x_{\mu}|T_{\mu}(S^{\ast}x_{\mu})-T_{\lambda}(S^{\ast}x_{\lambda})\rangle\\
 & =\Re\langle\lambda T_{\lambda}(S^{\ast}x_{\lambda})-\mu T_{\mu}(S^{\ast}x_{\mu})|T_{\mu}(S^{\ast}x_{\mu})-T_{\lambda}(S^{\ast}x_{\lambda})\rangle\\
 & \quad+\Re\langle J_{\lambda}(T)(S^{\ast}x_{\lambda})-J_{\mu}(T)(S^{\ast}x_{\mu})|T_{\mu}(S^{\ast}x_{\mu})-T_{\lambda}(S^{\ast}x_{\lambda})\rangle\\
 & \leq\Re\langle\lambda T_{\lambda}(S^{\ast}x_{\lambda})-\mu T_{\mu}(S^{\ast}x_{\mu})|T_{\mu}(S^{\ast}x_{\mu})-T_{\lambda}(S^{\ast}x_{\lambda})\rangle\\
 & \leq(\lambda+\mu)C,
\end{align*}
for a suitable constant $C>0$. This gives that $\left(x_{\lambda}\right)_{\lambda>0}$
converges to some $x\in H_{1}$ as $\lambda\to0.$ Moreover, $|J_{\lambda}(T)(S^{\ast}x_{\lambda})-S^{\ast}x|\leq|J_{\lambda}(T)\left(S^{\ast}x_{\lambda}\right)-S^{\ast}x_{\lambda}|+|S^{\ast}x_{\lambda}-S^{\ast}x|=\lambda|T_{\lambda}(S^{\ast}x_{\lambda})|+|S^{\ast}x_{\lambda}-S^{\ast}x|\to0$
as $\lambda\to0.$ Furthermore, since $\{T_{\lambda}(S^{\ast}x_{\lambda})\,|\,\lambda>0\}$
is bounded there exists a weakly convergent subsequence $(T_{\lambda_{n}}(S^{\ast}x_{\lambda_{n}}))_{n\in\mathbb{N}}$
with $\lambda_{n}\to0$ as $n\to\infty$ and we denote its weak limit
by $z\in H_{1}.$ Then, since $(J_{\lambda}(T)\left(S^{\ast}x_{\lambda}\right),T_{\lambda}(S^{\ast}x_{\lambda}))\in T$
for each $\lambda>0,$ the demi-closedness of $T$ implies $(S^{\ast}x,z)\in T.$
Since 
\[
ST_{\lambda}(S^{\ast}x_{\lambda})=y-x_{\lambda}\to y-x,
\]
we have that $Sz=y-x,$ which gives $(x,y-x)\in STS^{\ast}$ or equivalently
$(x,y)\in1+STS^{\ast}.$ 
\end{proof}

\section*{Acknowlodgement}

The author thanks Marcus Waurick for careful reading and fruitful
discussions.

\end{document}